\documentclass[leqno,12pt]{amsart}

\usepackage{amsmath,amstext,amssymb,amsopn,amsthm,mathrsfs}
\usepackage{enumerate}
\usepackage{graphicx}
\usepackage{multirow}
\usepackage{color}

\numberwithin{equation}{section}

\allowdisplaybreaks

\DeclareMathOperator{\cl}{cl}
\DeclareMathOperator{\shadow}{shw}

\newtheorem{theor}{Theorem}[section]
\newtheorem{bigthm}{Theorem}
\newtheorem{propo}[theor]{Proposition}
\newtheorem{lemma}[theor]{Lemma}

\newtheorem*{rem*}{Remark}
\newtheorem{rem}[theor]{Remark}

\def\pii{\tau}

\newcommand{\R}{\mathbb{R}^d_+}

\newcommand{\A}{\mathfrak{a}}

\begin{document}

\footnotetext{

	
Research of the first-named and the fourth-named authors was supported by the National Science Centre of Poland
within the project OPUS 2013/09/B/ST1/02057. The first-named author was additionally supported by the grant
OPUS 2017/27/B/ST1/01623 from the same institution.
}

\title[Maximal operators]
	{On non-centered maximal operators related to\\ a non-doubling and non-radial exponential measure}

\author[A.\ Nowak]{Adam Nowak}
\address{Adam Nowak \newline
			Institute of Mathematics, Polish Academy of Sciences \newline
      \'Sniadeckich 8,
      00--656 Warszawa, Poland
      }
\email{anowak@impan.pl}

\author[E.\ Sasso]{Emanuela Sasso}
\address{Emanuela Sasso \newline
		Dipartimento di Matematica, Universit\`a di Genova \newline
		Via Dodecaneso 35,
		Genova 16146, Italy
			}
\email{sasso@dima.unige.it}

\author[P.\ Sj\"ogren]{Peter Sj\"ogren}
\address{Peter Sj\"ogren \newline
			Mathematical Sciences, University of Gothenburg,
Mathematical Sciences, Chalmers University of Technology \newline
SE-412 96 G\"oteborg, Sweden
      }
\email{peters@chalmers.se}

\author[K.\ Stempak]{Krzysztof Stempak}	
\address{Krzysztof Stempak     \newline
     Wroc\l{}aw, Poland}
\email{krzysztof.stempak@pwr.edu.pl}
	
\begin{abstract}
We investigate mapping properties of non-centered Hardy-Littlewood maximal operators related to the exponential
measure $d\mu(x) = \exp(-|x_1|-\ldots-|x_d|)dx$ in $\mathbb{R}^d$. The mean values are taken over
Euclidean balls or cubes ($\ell^{\infty}$ balls) or diamonds ($\ell^1$ balls).
Assuming that $d \ge 2$, in the cases of cubes and diamonds we prove the $L^p$-boundedness for $p > 1$ and
disprove the weak type $(1,1)$ estimate. The same is proved in the case of Euclidean balls, under the restriction
$d \le 4$ for the positive part.
\end{abstract}

\maketitle

\section{Introduction and statement of the results} \label{sec:intro}

Let $d \ge 1$.
Consider a metric measure space $(\mathbb{R}^d,\rho,d\eta)$, with a Borel measure $\eta$ which is non-negative,
non-trivial and locally finite. The associated non-centered Hardy-Littlewood maximal operator is defined by
$$
M_{\eta}f(x) = \sup_{B_{\rho} \ni x} \frac{1}{\eta(B_{\rho})} \int_{B_{\rho}} |f|\, d\eta, \qquad x \in \mathbb{R}^d,
$$
where the supremum is taken over all open metric balls related to $\rho$ that contain $x$
and have strictly positive measure $\eta$.
Here $f$ is any Borel measurable function on $\mathbb{R}^d$. The centered variant of $M_{\eta}$, denoted by $M_{\eta}^c$,
arises by restricting the supremum to balls centered at $x$. Clearly, $M_{\eta}^cf \le M_{\eta}f$.
Furthermore, $M_{\eta}$ is trivially bounded on $L^{\infty}$.

When $\eta$ is doubling, the two maximal operators are comparable and satisfy the weak type $(1,1)$ estimate with respect
to $\eta$. The latter follows from a Vitali type covering lemma, cf.\ \cite[Chapter 2]{Hei}.
Then, by interpolation, $M_{\eta}$ and $M_{\eta}^c$ are bounded on $L^p(d\eta)$ for $p > 1$.

It is also well known, at least for the Euclidean distance $\rho$, that (see e.g.\ \cite[p.\,44]{Duo})
whatever the measure $\eta$ is, $M_{\eta}^c$ is always
of weak type $(1,1)$ with respect to $\eta$, thus also bounded on $L^p(d\eta)$ for $p>1$. The former is a consequence
of the Besicovitch-Morse covering lemma.
In dimension one the larger uncentered operator $M_{\eta}$ behaves in the same way (see \cite[p.\,45]{Duo}), that is,
it is of weak type $(1,1)$ and bounded on $L^p(d\eta)$, $p>1$, independently of the doubling property of $\eta$.
However, this is no longer true in general in higher dimensions.

One of the authors \cite{Sj} proved that for $d = 2$ (implicitly $d \ge 2$) and either
the Euclidean or the $\ell^{\infty}$ distance $\rho$, and the Gaussian measure $\eta$, the
weak type $(1,1)$ estimate for $M_{\eta}$ fails. Nevertheless, as shown by Forzani et al.\ \cite{FSSU}, the
$L^p$-boundedness for $p > 1$ in this case still holds, though the convenient interpolation argument is inapplicable.
Similar results for certain classes of rotationally invariant measures $\eta$ were established
in \cite{InSo,SaWe,SjSo,Va}, among others.
It is interesting to point out that there are radial measures $\eta$ for which $M_{\eta}$ is not even
weak type $(p,p)$ for any $p<\infty$, see \cite{In,InSo,Va}.

It should be mentioned that so far non-centered Hardy-Littlewood maximal operators for non-doubling measures
were studied in various settings and spaces also different from $\mathbb{R}^d$, for example in the framework
of cusped manifolds \cite{Li1,Li2}.

The main aim of this paper is to study the maximal operator $M_{\eta}$ when the distance $\rho$ is the Euclidean one
and for the particular exponential measure $\eta=\mu$,
$$
d\mu(x) = \exp\big(-|x_1|-\ldots - |x_d|\big)\, dx.
$$
Our motivation is to provide both methods and results in this model case where the measure is non-doubling and non-radial,
since the literature seems to lack a basic example of this kind. Only recently H.-Q.\ Li, Y.\ Wu and one of the authors \cite{LiSj}
considered $M_{\eta}$ essentially for $d\eta(x)=e^{x_1}dx$ in $\mathbb{R}^d$.
In this case the measure, in contrast with $\mu$, is neither finite nor even in each variable. Moreover, it has a simple structure
that makes the associated analysis relatively straightforward.

The measure $\mu$ is not radial in the sense of the Euclidean distance, nevertheless it is radial with respect to
the $\ell^1$ metric. Thus one might wonder whether, perhaps, the maximal operator behaves better when $\rho$ is the
seemingly better matching $\ell^1$ distance. This issue led us to study $M_{\mu}$ also when $\rho$ is the $\ell^1$
metric, as well as in the opposite extreme case where $\rho$ is the $\ell^{\infty}$ metric.

Denote by $M_{\mu}^{\mathcal{B}}$, $M_{\mu}^{\mathcal{Q}}$, $M_{\mu}^{\mathcal{D}}$ the maximal operators $M_{\mu}$
with the underlying $\ell^2$ or $\ell^{\infty}$ or $\ell^1$ metric, respectively. Note that the metric balls in the
first case are just the Euclidean balls $\mathcal{B}$, and in the second case the Euclidean cubes $\mathcal{Q}$ with sides
parallel to the coordinate axes. The third case is geometrically somewhat more complicated, and we call the metric balls
\emph{diamonds} $\mathcal{D}$ in this situation. Notice that in dimension $d=2$ the diamonds are simply rotated cubes
(or actually squares), but there is no similar relation in higher dimensions.

Our main result is the following theorem. We strongly believe it will be an inspiration for considering $M_{\eta}$
with more general non-radial and non-doubling $\eta$, and for further research in the future.
\begin{bigthm} \label{thm:main}
Let $d \ge 2$.
\begin{itemize}
\item[(A)] None of the maximal operators $M_{\mu}^{\mathcal{B}}$, $M_{\mu}^{\mathcal{Q}}$, $M_{\mu}^{\mathcal{D}}$
	is of weak type $(1,1)$.
\item[(B)] The operators $M_{\mu}^{\mathcal{Q}}$ and $M_{\mu}^{\mathcal{D}}$ are bounded on $L^p(d\mu)$ for $p>1$.
					The same is true for $M_{\mu}^{\mathcal{B}}$, provided that $d \le 4$.
\end{itemize}
\end{bigthm}

\begin{rem} \label{rem:restr}
The restriction $d \le 4$ in Theorem \ref{thm:main}(B), the case of $M_{\mu}^{\mathcal{B}}$, is caused by substantial
technical difficulties of geometrical nature in proving the result in dimensions $d=5$ and higher.
Nevertheless, we strongly believe that the result is true for any $d \ge 2$.
\end{rem}

When $d=1$, in view of what was said above, all the three maximal operators coincide and are of weak type $(1,1)$ and bounded on
$L^p(d\mu)$, $p>1$. Note that the latter readily implies Theorem \ref{thm:main}(B) for $M_{\mu}^{\mathcal{Q}}$. Indeed, due to
the product structure of the cubes $M_{\mu}^{\mathcal{Q}}$ can be controlled by a composition of the one-dimensional operators.

Theorem \ref{thm:main} reveals that the $L^p$ behavior of $M_{\mu}^{\mathcal{B}}$ and $M_{\mu}^{\mathcal{Q}}$ is
exactly the same as in case of their counterparts for the Gaussian measure \cite{Sj,FSSU}. In particular, we see that
the local doubling property (see Section \ref{sec:prep}), satisfied by $\mu$ but not by the Gaussian measure,
does not lead here to any improvement.

An interesting but technically quite complicated problem is to generalize Theorem \ref{thm:main} to Laguerre-type measures of the form
\begin{equation} \label{mua}
d\mu_{\alpha}(x) = \prod_{i=1}^d |x_i|^{\alpha_i} \exp\big(-|x_i|\big)\, dx,
\end{equation}
where $\alpha = (\alpha_1,\ldots,\alpha_d) \in (-1,\infty)^d$ is a fixed multi-parameter.
Clearly, the special choice $\alpha=(0,\ldots,0)$ gives $\mu$.
The restriction of the measure space $(\mathbb{R}^d,d\mu_{\alpha})$ to $(0,\infty)^d$
forms a natural environment for analysis related to the classical Laguerre operator.
Analysis of various objects in this context has already received considerable attention;
see for instance \cite{Di,NSS,Sa1,Sa2} and references given there.
Thus any knowledge about the non-centered Hardy-Littlewood maximal operator $M_{\mu_{\alpha}}$ or its
variants would be potentially useful.
For some negative results, see Remark \ref{rem:gen2} below, which says that $M_{\mu_{\alpha}}$ is not of weak type
$(1,1)$ when the underlying metric is either $\ell^2$ or $\ell^{\infty}$.

The remaining part of the paper is devoted to the proof of Theorem \ref{thm:main}.
The subsequent sections contain technical preliminaries, the proof of Theorem \ref{thm:main}(A) and
the proof of Theorem \ref{thm:main}(B), respectively.

\section{Technical preliminaries} \label{sec:prep}

Denote $\mathbb{R}^d_+ = (0,\infty)^d$, $d \ge 1$.
For brevity the restriction of $\mu$ to $\mathbb{R}^d_+$ will be denoted by the same symbol.
We write $|\cdot|_{q}$ for the $\ell^q$, $1 \le q \le \infty$, norm in $\mathbb{R}^d$,
$$
|x|_q = \bigg( \sum_{i=1}^d |x_i|^q \bigg)^{1/q} \quad \textrm{if} \quad q < \infty, \qquad
|x|_{\infty} = \max_{1\le i \le d} |x_i|.
$$
Of course, this norm generates
a metric $\rho_q$ both in $\mathbb{R}^d$ and $\mathbb{R}^d_+$. For $q=1,2,\infty$ we denote the families of open
balls in the metric measure spaces $(\mathbb{R}^d_+,\rho_q,d\mu)$ by $\mathcal{D}_+$, $\mathcal{B}_+$, $\mathcal{Q}_+$,
respectively. Notice that these are exactly diamonds, Euclidean balls and cubes, respectively, centered in and intersected with
$\mathbb{R}^d_+$.

Bring in the non-centered Hardy-Littlewood maximal operator
$$
M_{\mu}^{\mathcal{B}_+}f(x) = \sup_{x \in B \in \mathcal{B}_+} \frac{1}{\mu(B)} \int_{B} |f|\, d\mu, \qquad
	x \in \mathbb{R}^d_+,
$$
and analogously $M_{\mu}^{\mathcal{Q}_+}$ and $M_{\mu}^{\mathcal{D}_+}$.
The following elementary result shows that proving Theorem \ref{thm:main} can be reduced to a
similar analysis for $M_{\mu}^{\mathcal{B}_+}$, $M_{\mu}^{\mathcal{Q}_+}$ and $M_{\mu}^{\mathcal{D}_+}$.
\begin{propo} \label{prop:red}
Let $d \ge 1$ and $p > 1$ be fixed.
The operator $M_{\mu}^{\mathcal{B}}$ is bounded on $L^p(\mathbb{R}^d,d\mu)$ (is weak type $(1,1)$ with respect to
$(\mathbb{R}^d,d\mu)$) if and only if $M_{\mu}^{\mathcal{B}_+}$ is bounded on $L^p(\mathbb{R}^d_+,d\mu)$
(is weak type $(1,1)$ with respect to $(\mathbb{R}^d_+,d\mu)$).

The same relations hold between $M_{\mu}^{\mathcal{Q}}$ and $M_{\mu}^{\mathcal{Q}_+}$, as well as between
$M_{\mu}^{\mathcal{D}}$ and $M_{\mu}^{\mathcal{D}_+}$.
\end{propo}

\begin{proof}
This is a consequence of the symmetries involved.
Use either the even (more precisely even with respect to each coordinate axis) extension to
$\mathbb{R}^d$ of $f_+$ on $\mathbb{R}^d_+$ or,
for the other implication, the decomposition of $f$ on $\mathbb{R}^d$ into its symmetric components which are either even or
odd with respect to each coordinate axis.
\end{proof}

Thus, from now on, we focus on the restricted operators
$M_{\mu}^{\mathcal{B}_+}$, $M_{\mu}^{\mathcal{Q}_+}$ and $M_{\mu}^{\mathcal{D}_+}$.
This is a crucial reduction from a technical point of view, since in $\mathbb{R}_+^d$ the measure $\mu$ has a simpler
analytic structure than in $\mathbb{R}^d$ (no absolute values involved).
From now on $\mu$ will denote the restriction of the measure with density $\exp(-|x|_1)$ to $\mathbb R^d_+$.

In what follows we shall write $X \lesssim Y$ with $Y>0$ to indicate that $X \le C Y$ with a constant $C>0$
depending only on the dimension and on $p$ in the proofs of $L^p$ estimates, and also on $\alpha$ in
Remarks \ref{rem:gen} and \ref{rem:gen2}. We write $X\simeq Y$ when simultaneously $X \lesssim Y$ and $Y \lesssim X$.

We will occasionally refer to the strong maximal operator in Euclidean space with Lebesgue measure. It is defined as
\begin{equation} \label{strong}
 M_{\mathrm{str}} f(x) = \sup \frac 1{|R|}\, \int_R |f(y)|\,dy,
\end{equation}
where the supremum is taken over all rectangles with edges parallel with the coordinate axes and containing $x$.
It is well known that $M_{\mathrm{str}}$ is bounded on $L^p(dx)$ for $1<p \le \infty$,
as seen by iterating the one-dimensional estimate.

We shall use the following notation for $\ell^1$, $\ell^2$ and $\ell^{\infty}$ balls in $\mathbb{R}^d_+$.
For $x \in \mathbb{R}^d_+$ and $r > 0$
\begin{align*}
D(x,r) & = \big\{ y \in \mathbb{R}^d_+ : |x-y|_1 < r \big\}, \\
B(x,r) & = \big\{ y \in \mathbb{R}^d_+ : |x-y|_2 < r \big\}, \\
Q(x,r) & = \big\{ y \in \mathbb{R}^d_+ : |x-y|_{\infty} < r \big\}.
\end{align*}
Euclidean balls in all of $\mathbb{R}^d$ will be written as
$$
\mathbf{B}(x,r) = \big\{ y \in \mathbb{R}^d : |x-y|_2 < r \big\}.
$$
Further, we denote
\begin{align*}
\mathbf{1} & = (1,\ldots,1) \in \mathbb{R}^d_+, \\
\Sigma^{d-1}_+ & = \big\{ x \in \mathbb{R}^d_+ : |x|_1 = 1\big\},\\
a \vee b & = \max(a,b),  \\
a \wedge b & = \min(a,b).
\end{align*}

The measure $\mu$ is not doubling in $(\mathbb{R}^d_+,\rho_q,d\mu)$, $q=1,2,\infty$; nevertheless it is locally doubling
in the following sense.
\begin{lemma} \label{lem:locd}
Let $d \ge 1$. Given $R>0$, there exists a constant $C_R>0$ such that
\begin{equation} \label{locd}
\mu\big(Q(x,2r)\big) \le C_R \; \mu\big( Q(x,r) \big), \qquad x \in \mathbb{R}^d_+, \quad 0 < r \le R.
\end{equation}
The same holds if $Q$ above is replaced either by $B$ or by $D$.
\end{lemma}

\begin{proof}
This is elementary, since in any of the balls considered the density of $\mu$ varies at most by a factor depending only on $R$.
\end{proof}

Let $d \ge 1$.
We now give sharp estimates for the measure of large cubes, balls and diamonds provided that they are disjoint with
the boundary of $\mathbb{R}^d_+$.
Consider a ball in one of the three metrics  $\ell^\infty ,\:\ell^2,\:  \ell^1 $
with center  $x \in \mathbb{R}^d_+$ and radius $r$ satisfying $1 \le r \le \min_{1 \le i \le d} x_i$.
We select a point $z_q = z_q(x,r)$ in the closure of this ball where $|\cdot|_1$ is minimal,
i.e., the density of $\mu$  is maximal, as follows:
\begin{align*}
z_\infty & = z_{\infty}(x,r) =  x-r\,\mathbf{1},\\
z_2 &= z_2(x,r) = x-\frac{r}{\sqrt{d}}\,\mathbf{1},\\
z_1 &= z_1(x,r) = x-\frac r {d}\,\mathbf{1}.
\end{align*}
Notice that $z_\infty$ and $z_2$ are unique points with this minimizing property, but $z_1$ is not.

\begin{lemma} \label{lem:balls}
Let  $x \in \mathbb{R}^d_+$ and $ 1 \le r \le  x_i,\: i = 1,\dots, d$. Then the balls $Q(x,r)$, $B(x,r)$ and $D(x,r)$
are contained in $\mathbb{R}^d_+$ and
\begin{align*}
\mu\big( Q(x,r)\big) & \simeq \exp(-|z_\infty|_1), \\
\mu\big( B(x,r)\big) & \simeq \exp(-|z_2|_1)\, r^{(d-1)/2}, \\
\mu\big( D(x,r)\big) & \simeq \exp(-|z_1|_1)\, r^{d-1}.
\end{align*}
The implicit constants here depend only on $d$.
\end{lemma}

\begin{proof}
The inclusions follow, since if $y$ is in one of the balls, then $|y_i-x_i| <r$ for each $i$, so that $y_i>0$.

The estimate for cubes is straightforward. One has
\begin{align*}
\mu\big( Q(x,r) \big) & = \int_{Q(x,r)} \exp(-|y|_1)\, dy =
	\prod_{i=1}^d \Big( e^{-(x_i-r)} - e^{-(x_i+r)} \Big) \\
& \simeq \exp(-|x|_1) e^{rd} = \exp(-|z_\infty|_1).
\end{align*}

To deal with the case of Euclidean balls, observe that any point in $B(x,r)$ can be written as
$z_2 + \frac s {\sqrt{d}}\,\mathbf{1} + y$, where $s>0$ and $y\perp\mathbf{1}$. Using the expression for $z_2$,
we see that this point is in  $B(x,r)$ precisely when $(r-s)^2 + |y|^2 < r^2$ or equivalently $|y|< \sqrt{2rs-s^2}$ and $0<s<2r$.
We now integrate in $y$ in a hyperplane orthogonal to $\mathbf{1}$ and then in $s$, taking the density of $\mu$ into account.
For the upper estimate, we simply write
\begin{align*}
 \mu\big( B(x,r)\big)  & \lesssim \int_0^{2r} \exp\big(-|z_2|_1 -\sqrt{d}\,s\big)\,(2rs-s^2)^{(d-1)/2}\,ds \\
	& \lesssim  \exp(-|z_2|_1)\, \int_0^{\infty} \exp\big(-\sqrt{d}\,s\big)\,(rs)^{(d-1)/2}\,ds   \simeq   \exp(-|z_2|_1)\, r^{(d-1)/2}.
\end{align*}
To obtain the lower estimate, we observe that $2rs-s^2 > rs$ for $0<s<r$ and argue similarly.
Since $r \ge 1$, we get
\begin{equation*}
 \mu\big( B(x,r)\big)  \gtrsim \int_0^{r} \exp\big(-|z_2|_1 -\sqrt{d}\,s\big)\,(rs)^{(d-1)/2}\,ds
	\simeq   \exp(-|z_2|_1)\, r^{(d-1)/2}.
\end{equation*}

As for the diamonds, note that $|z_1|_1 = |x|_1-r$.
For $s>0$ the diameter of the intersection of ${D(x,r)}$ with the hyperplane $\{y: |y|_1 = |x|_1-r +s\}$ is $\mathcal O(r)$.
Integrating as before, we obtain the upper estimate.

On the other hand, consider the following set
\begin{equation*}
\left\{x-\frac{r}{d}\,\mathbf{1}+\frac s {d}\,\mathbf{1}+y:\; 0<s<r/2,\;\; y\perp\mathbf{1},\; \; |y_i|<\frac{r}{2d}
	\;\; \mathrm{for} \;\; i=1,\dots,d \right\}.
\end{equation*}
The $\ell^1$ distance from $x$ to a point in this set is
\begin{equation*}
\sum_{i=1}^d \left|-\frac r {d}+\frac s {d}+y_i\right| =  \sum_{i=1}^d \left(\frac r {d}-\frac s {d}-y_i\right)
 \le  {r-s} < r.
\end{equation*}
Thus ${D(x,r)}$ contains the set, and the lower estimate follows by integration.
\end{proof}

\begin{rem} \label{rem:gen}
Lemmas \ref{lem:locd} and \ref{lem:balls} can be generalized to the space $(\mathbb{R}^d_+,\rho_q,d\mu_{\alpha})$, where
$q \in \{1,2,\infty\}$ and $\mu_{\alpha}$ is the restriction of the measure defined in \eqref{mua}.
This means that $\mu_{\alpha}$ is locally doubling (but not doubling) in the context of this space.
Moreover,
$$
\mu_{\alpha}\big( E_q(x,r) \big) \simeq x_1^{\alpha_1}\cdot \ldots \cdot x_d^{\alpha_d} \exp(-|x|_1)\,
	r^{(d-1)/q} e^{rd^{1-1/q}}
$$
uniformly in $x \in \mathbb{R}^d_+$ and $1 \le r \le \min_{1\le i \le d} x_i$; here $E_q(x,r)$ is the open ball
in $(\mathbb{R}^d_+,\rho_q)$ centered at $x$ and of radius $r$.

Proposition \ref{prop:red} can also be generalized in a similar spirit.
\end{rem}

We now pass to the proof of Theorem \ref{thm:main}. It is worth indicating that the
radiality of $\mu$ with respect to the $\ell^1$ norm will be heavily exploited, often implicitly,
throughout our reasonings.

\section{Proof of {Theorem \ref{thm:main}}(A)} \label{sec:proofa}

In this section we prove Theorem \ref{thm:main}(A) working with the operators restricted to $\mathbb{R}^d_+$,
see Proposition \ref{prop:red}.
The cases of $M_{\mu}^{\mathcal{Q}_+}$ and $M_{\mu}^{\mathcal{B}_+}$
will be treated together, since the argument is essentially the same. This argument has the advantage that it can
be rather easily generalized to cover $M_{\mu_{\alpha}}^{\mathcal{Q}_+}$ and $M_{\mu_{\alpha}}^{\mathcal{B}_+}$
(analogues of $M_{\mu}^{\mathcal{Q}_+}$ and $M_{\mu}^{\mathcal{B}_+}$ for the measure $\mu_{\alpha}$),
see Remark \ref{rem:gen2} below.
Unfortunately, this argument does not apply to $M_{\mu}^{\mathcal{D}_+}$ since it uses essentially the non-radiality
of the measure with respect to the norm. Therefore, we give a different argument for $M_{\mu}^{\mathcal{D}_+}$,
but the question of its generalization to $M_{\mu_{\alpha}}^{\mathcal{D}_+}$ seems to be technically difficult
and remains open.

\subsection*{Proof of Theorem \ref{thm:main}(A), the cases of $\mathbf{M_{\mu}^{\mathcal{Q}}}$
and $\mathbf{M_{\mu}^{\mathcal{B}}}$}
We first consider the case $d=2$ and then indicate the changes needed for $d\ge3$. We begin with the operator
$M_{\mu}^{\mathcal{Q}_+}$. Let $Q_s$, $s\ge1$, denote the square centered at $(s,s)$ and of `radius' $s\slash2$.
Further, let $\widehat{Q_s}$ be the union of all squares obtained by moving $Q_s$ (or rather its center) along the
line segment $\Delta_s$ which is the intersection of $\frac12Q_s$ with the line $x+y=2s$, see Figure~\ref{fig1}.
\begin{figure}[ht]
\includegraphics[height=0.5\textwidth]{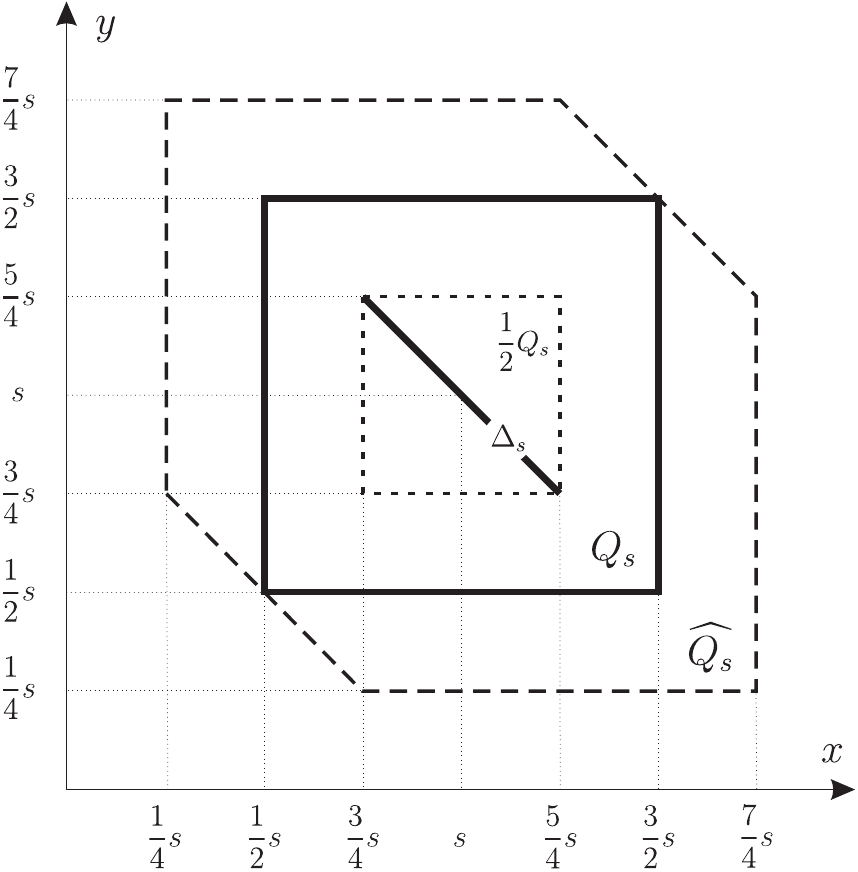}
\caption{Counterexample for $M_{\mu}^{\mathcal{Q}_+}$ in dimension $d=2$.}
\label{fig1}
\end{figure}
Assuming \textit{a contrario} that $M_{\mu}^{\mathcal{Q}_+}$ is of
weak type $(1,1)$, we claim that
 \begin{equation}\label{4z}
\mu(\widehat{Q_s})\lesssim \mu(Q_s), \qquad s\ge2.
\end{equation}
To see this, take $(x_0,y_0)\in \widehat{Q_s}$ and find a square $Q^0=Q((x',y'),s\slash2)$
with center on $\Delta_s$ and of side length $s$, such that $(x_0,y_0)\in Q^0$.
It is clear that $\frac12Q_s\subset Q^0$,
and by Lemma \ref{lem:balls} $\mu(Q^0) = \mu(Q_s) \simeq e^{-s}$.
Thus, for the $L^1$-normalized function $\widetilde{\chi}=\frac1{\mu(\frac12Q_s)}\chi_{\frac12Q_s}$ one has
$$
M_{\mu}^{\mathcal{Q}_+}\widetilde{\chi}(x_0,y_0)\ge \frac1{\mu(Q^0)}\int_{Q^0} \widetilde{\chi} \,d\mu = \frac1{\mu(Q_s)}.
$$
We conclude that
$$
\widehat{Q_s}\subset \Big\{(x,y): M_{\mu}^{\mathcal{Q}_+}\widetilde{\chi}(x,y)\ge \frac1{\mu(Q_s)}\Big\},
$$
hence \eqref{4z} follows. On the other hand, since $\widehat{Q_s}$ contains the
rectangle $R_s$ with basis $\Delta_s - (\frac{s}2,\frac{s}2)$ and height $\sqrt{2}s$,
we have
$$
\mu(\widehat{Q_s}) \ge \int_{R_s}  e^{-(x+y)}\, dx dy
	\gtrsim s \int_{s}^{3s} e^{-r}\, dr \gtrsim s e^{-s}.
$$
For large $s$ this contradicts \eqref{4z} since, as already noted, $\mu(Q_s)\simeq e^{-s}$.

We now continue with the operator $M_{\mu}^{\mathcal{B}_+}$ in dimension $d=2$.
Let $B_s$, $s\ge1$, denote the ball with center at $(s,s)$ and radius $s/2$ (thus $B_s$ is a usual Euclidean disc),
and let $\widehat{B_s}$ be the union of all discs obtained by moving $B_s$ (or rather its center) along the line segment
$\widetilde{\Delta_s}$ which is the intersection of $\frac12B_s$ with the line $x+y=2s$, see Figure~\ref{fig2}.
\begin{figure}[ht]
\includegraphics[height=0.5\textwidth]{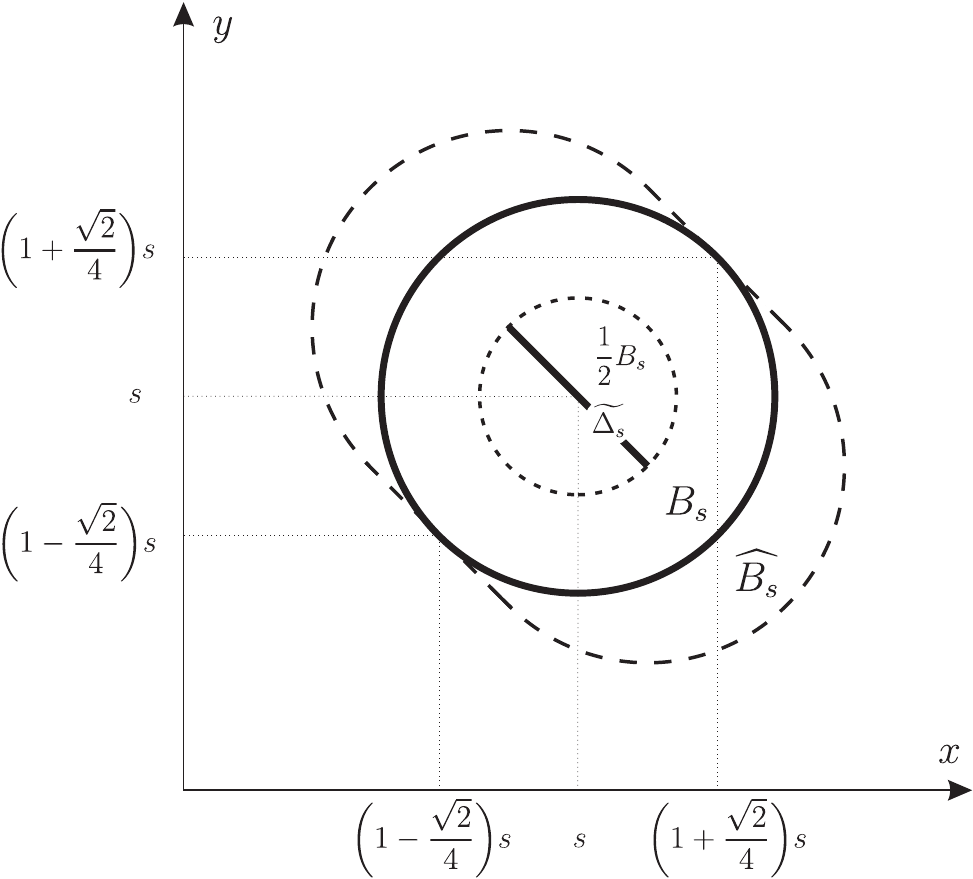}
\caption{Counterexample for $M_{\mu}^{\mathcal{B}_+}$ in dimension $d=2$.}
\label{fig2}
\end{figure}
Again, assuming \textit{a contrario} that $M_{\mu}^{\mathcal{B}_+}$ is of weak type $(1,1)$, we claim that
\begin{equation}\label{4z2}
\mu(\widehat{B_s})\lesssim \mu(B_s), \qquad s\ge2.
\end{equation}
The argument is similar to that for squares.
 Lemma \ref{lem:balls} yields
 $$
\mu(B_s) = \mu\big(B((x',y'),s\slash2)\big) \simeq \sqrt{s} \,e^{-s(2-\frac{\sqrt2}2)},
	\qquad s \ge 2, \quad (x',y')\in \widetilde{\Delta_s}.
$$
Since $\widehat{B_s}$ contains the rectangle $\widetilde{R_s}$ with
basis $\widetilde{\Delta_s}-(\frac{s}{2\sqrt{2}},\frac{s}{2\sqrt{2}})$ and height $s$ (that contains $\frac{1}2 B_s$), we have
$$
\mu(\widehat{B_s}) \ge \int_{\widetilde{R_s}} e^{-(x+y)}\,dx dy
\gtrsim s\int_{(2-\frac{\sqrt2}2)s}^{(2+\frac{\sqrt2}2)s}e^{-r}\,dr
\gtrsim s e^{-(2-\frac{\sqrt2}2)s},
$$
which for large $s$ contradicts \eqref{4z2}.

We pass to explaining the changes necessary for $d\ge3$. Let $Q_s$, $s\ge1$,
denote the cube centered at $s\mathbf{1}$, of side length $s$,
and let $\widehat{Q_s}$ be the union of all cubes emerging from moving the center of $Q_s$ along the hypersegment
$\Delta_s$ obtained by intersecting $\frac12Q_s$ with the hyperplane $x_1+\ldots+x_d=ds$.
With the present notation the justification of \eqref{4z}, assuming \textit{a contrario} the weak type $(1,1)$ of
$M_{\mu}^{\mathcal{Q}_+}$, is analogous to that for the case $d=2$ and involves the estimate (see Lemma \ref{lem:balls})
$$
 \mu(Q_s) = \mu\big(Q(x',s/2)\big)\simeq e^{-ds\slash2}, \qquad s\ge2,\quad x'\in \Delta_s.
$$
Now \eqref{4z} is contradicted for large $s$ by
$$
\mu(\widehat{Q_s})\gtrsim s^{d-1}e^{-ds\slash2}.
$$
To justify the last estimate, observe that $\widehat{Q_s}$ contains the hyperprism $R_s$
with basis $\Delta_s-\frac{s}2 \textbf{1}$ and height $\sqrt{d}s$. Then
$$
\mu(\widehat{Q_s}) \ge \int_{R_s} e^{-|x|_1}\, dx \gtrsim
	s^{d-1} \int_{\frac{ds}2}^{\frac{3ds}2} e^{-r}\, dr \gtrsim s^{d-1}e^{-ds/2},
		\qquad s \ge 1.
$$

Similarly, let $B_s$, $s\ge1$,  denote the ball with center at $s\textbf{1}$ and radius $s\slash2$,
and let $\widehat{B_s}$ be the union of balls emerging from moving the center of $B_s$ along the hypersegment $\Delta_s$
obtained by intersecting $\frac12B_s$ with the hyperplane
$x_1+\ldots+x_d=ds$. Assuming again \textit{a contrario} the weak type $(1,1)$ of $M_{\mu}^{\mathcal{B}_+}$, we prove
\eqref{4z2} in a way  analogous to that for the case $d=2$ with the estimate
$$
 \mu\big(B(x',s\slash2)\big)\simeq s^{\frac{d-1}2}e^{-s(d-\frac{\sqrt d}2)},
 	\qquad s\ge2,\quad x'\in \widetilde{\Delta_s}
$$
included. Let $\widetilde{R_s}$ be the cylinder with basis $\widetilde{\Delta_s}-\frac{s}{2\sqrt{d}}$
and height $s$ that includes $\frac12B_s$. Since $\widetilde{R_s}\subset \widehat{B_s}$, we have
$$
\mu(\widehat{B_s}) \ge \int_{\widetilde{R_s}} e^{-|x|_1}\,dx\gtrsim
 s^{d-1}\int_{(d-\frac{\sqrt d}2)s}^{(d+\frac{\sqrt d}2)s}e^{-r}\,dr
 \gtrsim s^{d-1}e^{-(d-\frac{\sqrt d}2)s}.
$$
For large $s$, this contradicts $\mu(B_s)\simeq s^{\frac{d-1}2}e^{-s(d-\frac{\sqrt d}2)}$.
This finishes the proof.
\qed

\begin{rem} \label{rem:gen2}
In view of Remark \ref{rem:gen}, the above proof extends in a
straightforward manner to the context of the measure $\mu_{\alpha}$ given in \eqref{mua}. Consequently,
$M_{\mu_{\alpha}}^{\mathcal{Q}_+}$ and $M_{\mu_{\alpha}}^{\mathcal{B}_+}$ are not weak type $(1,1)$.
\end{rem}

\subsection*{Alternative condensed version of the proof of Theorem \ref{thm:main}(A), the cases of $\mathbf{M_{\mu}^{\mathcal{Q}}}$
and $\mathbf{M_{\mu}^{\mathcal{B}}}$}

Consider first ${M_{\mu}^{\mathcal{Q}_+}}$. With $s>1$, we
choose $0 \le f \in L^1(d\mu)$ so that the measure $f d\mu$ is a close approximation of the Dirac measure
$\delta_{2s\mathbf{1}}$.
The cube $Q(2s\mathbf{1}+y,s)$  will contain the point $2s\mathbf 1$ if
$y\perp \mathbf 1$  and $|y|_\infty <s$, and this cube is contained in $\mathbb{R}^d_+$.
Then any point $x \in Q(2s\mathbf{1}+y,s)$ will satisfy
\begin{equation}\label{lowerMQ}
 M_{\mu}^{\mathcal{Q}_+} f(x)\gtrsim \mu\big(Q(2s\mathbf{1}+y,s)\big)^{-1} \simeq \exp{ \left(|z_\infty|_1\right)},
\end{equation}
where we applied Lemma \ref{lem:balls}, and $z_\infty = z_\infty(2s\mathbf{1}+y,s) = s\mathbf{1}+y$. Notice that
$|z_\infty|_1 = ds$ does not depend on $y$. The union of these cubes taken  over all admissible points $y$
will contain the set
\begin{equation*}
\big \{ \sigma \mathbf{1} + y:\: s< \sigma < 2s,\;   y\perp \mathbf 1, \; |y|_\infty <s   \big\},
\end{equation*}
 whose $\mu$ measure is at least
$c\exp{ \left(-|z_\infty|_1\right)}\,s^{d-1} $.
Since \eqref{lowerMQ} holds in this set, the weak type $(1,1)$ inequality is violated for large $s$.

In the case of  ${M_{\mu}^{\mathcal{B}_+}}$, we proceed similarly, with the same $f$ but with the balls
$B(2s\mathbf 1+y,s)$ instead of the cubes. In view of  Lemma \ref{lem:balls}, the estimate  \eqref{lowerMQ} will now read
$ M_{\mu}^{\mathcal{B}_+} f(x) \gtrsim  \mu(B(2s\mathbf{1}+y,s))^{-1} \simeq
\exp{ \left(|z_2|_1\right)} \,s^{(1-d)/2}$, where  $z_2 = z_2(2s\mathbf{1}+y,s)$.
The measure of the union of the balls will be at least $c\exp{ \left(-|z_2|_1\right)}\,s^{d-1}$.
These two estimates together disprove the weak type inequality.
\qed

\subsection*{Proof of Theorem \ref{thm:main}(A), the case of $\mathbf{M_{\mu}^{\mathcal{D}}}$}
Fixing a large $N>0$,
we now let $f d\mu$ approximate $\delta_{(0,\dots,0,N)}$ (cf.\ the second proof for $M_{\mu}^{\mathcal{Q}_+}$).

Let $\xi\in \mathbb R^d_+$ with $|\xi|_1 < N$, and write $s = |\xi |_1$.
To estimate $M^{\mathcal{D}_+}_{\mu} f(\xi)$ from below,
we introduce a closed diamond  $D=\{x\in \mathbb R^d_+:|x-c|_1\leq M\}$
with $c_i=\xi_i$ for $i<d$ and $c_d=\xi_d+M$.
Here $M > N+s$. Then the points
 $\xi$ and $(0,\dots,0,N)$ are both in $D$, and  $|x|_1 \ge s$ if $x \in D$.
Since $x_d\ge \xi_d$ for all points $x\in D$, one has for $h>0$
\begin{align*}
  D\cap \{x:\:|x|_1=s+h\}\subset &
\left\{x \in \mathbb R^d_+:\:|x|_1=s+h,\;\; \xi_d\le x_d \right\} \\
\subset & \Big\{x \in \mathbb R^d_+:\:  x_d = s+h- \sum_1^{d-1} x_i,\; \;\;  \sum_1^{d-1} x_i \le s+h- \xi_d  \Big\}
\end{align*}
and the $(d-1)$-dimensional area of the last set here is $\mathcal{O}((s+h-\xi_d)^{d-1})$, as seen by projecting onto
$\mathbb{R}^{d-1}$.
Thus
\[
\mu(D) \lesssim \int_0^\infty (s+h-\xi_d)^{d-1} e^{-s-h}\,dh \simeq
(1+s-\xi_d)^{d-1}\, e^{-s}.
\]
This implies that
$M^{\mathcal{D}_+}_{\mu} f(\xi)\gtrsim {e^s}/(1+s-\xi_d)^{d-1}$; observe that $s-\xi_d = \sum_1^{d-1} \xi_i$.

Next we choose the level $\lambda= {N^{1-d}\,e^N}$ and examine when
 $M^{\mathcal{D}_+}_{\mu} f(\xi)\gtrsim \lambda $. This occurs if
$1+(s-\xi_d)^{d-1} \lesssim N^{d-1} e^{s-N}$, in particular if
\begin{equation*}
1 < N^{d-1} e^{s-N} \quad \mathrm{i.e.} \quad s > N-(d-1)\log N   \quad \mathrm{and}  \quad
  (s-\xi_d)^{d-1} < N^{d-1} e^{s-N}.
\end{equation*}
To find points   $\xi\in \mathbb R^d_+$  satisfying these two inequalities, we  fix
\[
 |\xi|_1 = s \in \big(N-(d-1)\log N,\: N\big).
\]
We can then choose any  $\xi_i \in \left( 0,\: d^{-1}\,N\,e^{(s-N)/(d-1)} \right)$,
$i= 1,\dots, d-1$, and set $\xi_d = s - \sum_1^{d-1} \xi_i$.
Indeed, for such  points $\xi$ the first inequality is clear, and the second one follows because
\begin{equation*}
s-\xi_d = \sum_1^{d-1}  \xi_i  < N\,e^{(s-N)/(d-1)} < s.
\end{equation*}
Here the last inequality assures that
$\xi_d$ is positive, and it holds since  $s<N$ implies $e^s\,s^{1-d} < e^N\,N^{1-d}$ for large $s$ and $N$.

Keeping still  $s$ fixed, we see that the $(d-1)$-dimensional measure of the set of points  $\xi$ thus obtained is of order
of magnitude $ N^{d-1} e^{s-N}$. Varying then $s$, we conclude that the $\mu$-measure of the set of all points $\xi$ obtained
is greater than constant times
\[
 \int_{N-(d-1)\log N}^N N^{d-1} e^{s-N}e^{-s}\,ds \simeq N^{d-1}\,e^{-N}\log N  = \frac{\log N}\lambda.
\]
For large $N$, this contradicts the weak-type $(1,1)$ boundedness of $M^{\mathcal{D}_+}_{\mu}$.
\qed

\section{Proof of {Theorem \ref{thm:main}}(B)} \label{sec:proofb}

As remarked in Section \ref{sec:intro}, the case of $M_{\mu}^{\mathcal{Q}}$ in Theorem \ref{thm:main}(B)
is an immediate consequence of the one-dimensional result. The remaining two cases are much less straightforward and
will be treated subsequently. We shall work with the operators restricted to $\mathbb{R}^d_+$, see Proposition \ref{prop:red}.
We make the following two preliminary reductions in proving the $L^p$-boundedness of
$M_{\mu}^{\mathcal{D}_+}$ and ${M_{\mu}^{\mathcal{B}_+}}$.

\medskip

\noindent \textbf{Reduction 1.} We may consider only diamonds (elements of $\mathcal{D}_+$) or balls (elements of $\mathcal{B}_+$)
with radii bounded from below by any fixed positive constant, due to the local doubling property of $\mu$, see Lemma \ref{lem:locd}.

\medskip

\noindent \textbf{Reduction 2.} Among diamonds or balls remaining after Reduction 1, we may consider only those
not intersecting $t \Sigma_+^{d-1} = \{x \in \R : |x|_{1} = t \}$ for $0<t\le c$ with $c>2$ arbitrary and fixed, since otherwise
they have measures bounded from below (and above) by a positive constant.

\medskip

We first consider the simpler case $M^{\mathcal{D}_+}_{\mu}$. The reasoning in case of $M^{\mathcal{B}_+}_{\mu}$ is
more sophisticated, because of the geometry of the balls in $\mathbb{R}^d_{+}$, especially those touching the boundary
of $\mathbb{R}^d_+$.

\subsection*{Proof of Theorem \ref{thm:main}(B), the case of $\mathbf{M_{\mu}^{\mathcal{D}}}$}
Our aim is to prove that $M_{\mu}^{\mathcal{D}_+}$ is bounded on $L^p(\mathbb{R}^d_+,d\mu)$ for $1<p<\infty$.
Recall that diamonds in $\mathbb{R}^d_+$ are denoted
\begin{align*}
D(z,r) & = \big\{ y \in \mathbb{R}^d_+ : |z-y|_1 < r \big\}.
\end{align*}
Here $r > 0$, and $z$ will always be in  $\mathbb{R}^d_+$.

For each $x  \in \mathbb{R}^d$ we denote   $x_0 = \sum_1^d x_j$. Then
$$
\Pi_t = \{x \in \mathbb{R}^d:\: x_0 = t \}
$$
is a hyperplane for each $t \in \mathbb{R}$, and we write $\lambda_t$ for Lebesgue measure in $\Pi_t$.
Further,   $x_t$ will for $t>0$ denote the orthogonal projection on  $\Pi_t$ of  any point $x$.

Let $f$ be a nonnegative function in  $ L^1(d\mu)$, which we extend by $0$ in $ \mathbb{R}^d \setminus \mathbb{R}^d_+$.
We want to estimate $M_{\mu}^{\mathcal{D}_+}f$ at a point $\xi \in \mathbb{R}^d_+$.
So we take a diamond  ${D} = {D}(z,r)$ with $z \in \mathbb{R}^d_+$ and such that $\xi \in D$,
and estimate the mean
$$
\frac{1}{\mu(D)}\, \int_{D} f(y)\,d\mu(y).
$$
Reductions 1 and 2 allow us to assume that the quantities $r$ and $z_0 - r > 2$ are large.
It will be convenient to write $b = z_0-r$, which indicates the ``bottom'' of the diamond.

Denoting slices of $D$ as  $D_t = D \cap \Pi_t$, we can write this mean as
\begin{equation}\label{mean}
  \frac{1}{\mu(D)}\, \int_{b}^{b+2r}\, e^{-t}\, \int_{D_t} f\,d\lambda_t\,dt.
\end{equation}
The inner integral here will be estimated in terms of a $(d-1)$-dimensional maximal operator.
We define $V$ as the set consisting of the $d$-dimensional vector $v = (1,-1,0,\dots,0)$
and all the vectors obtained from $v$ by permuting the coordinates.

\begin{propo}\label{sophi}
For each $t \in (b,b+2r)$ there exists a $(d-1)$-dimensional parallelepiped $P_t \subset \Pi_t$
containing $D_t$ and containing $\xi_{t}$ such that
$$
\lambda_t(P_t) \lesssim \big[1+(t-b)\vee (\xi_0-b)\big]^{d-1}\, e^{b}\, \mu(D)
$$
and whose edges are all parallel to vectors in $V$.
\end{propo}

Before proving this proposition, we use it to finish the proof of the $L^p(d\mu)$-boundedness of $M_{\mu}^{\mathcal{D}_+}$.
Here $1<p<\infty$.

In the iterated integral in \eqref{mean}, we extend the inner integration to $P_{t}$ and insert the factor
$$
\frac{ [1+(t-b)\vee (\xi_0-b)]^{d-1}\, e^{b}\, \mu(D)}{\lambda_t(P_t)} \gtrsim 1.
$$

Thus  \eqref{mean} is controlled by
$$
\int_{b}^{b+2r}\: e^{-t+b}\, [1+(t-b)\vee (\xi_0-b)]^{d-1}\, \frac{1}{\lambda_t(P_t)}\,
\int_{P_{t}} f\,d\lambda_t\,dt.
$$

The mean over $P_{t}$ here can be estimated in terms of the non-centered maximal operator $\mathcal M_t$ in $\Pi_{t}$
associated with parallelepipeds having edges with directions from $V$, evaluated at $\xi_t$.
So the iterated integral is at most
\begin{equation}\label{integr}
 \int_{b}^{b+2r}\: e^{-t+b}\,\big[1+(t-b)\vee (\xi_0-b)\big]^{d-1}\, \mathcal M_tf(\xi_t)\,dt.
\end{equation}

We consider the exponent $-t+b$ here. Since $\xi_0>b$ and $t>b$, we have
\begin{align*}
-t+b & = \frac{\xi_0}{p} - \frac{t}{p} -  \frac{1}{p'}\,(t-b) - \frac{1}{p}\,(\xi_0-b) \\
&\le \frac{\xi_0}{p} - \frac{t}{p}- \Big(\frac{1}{p}\wedge \frac{1}{p'}\Big)\,\big[(t-b) \vee (\xi_0-b)\big] \\
&\le  \frac{\xi_0}{p} - \frac{t}{p}- c\,\big[(t-b) \vee (\xi_0-b)\big] -c\,|\xi_0-t|
\end{align*}
with $c = c(p) > 0$; in the last step we used the simple fact that $(t-b) \vee (\xi_0-b) \ge |\xi_0-t|$.
After inserting this estimate in the integral \eqref{integr}, we can delete the factors \\
$e^{-c\,[(t-b) \vee (\xi_0-b)]}\,\big[1 + (t-b) \vee (\xi_0-b)\big]^{d-1}$, and thus estimate \eqref{integr} by constant times
$$
 \int_{b}^{b+2r}\:e^{{\xi_0}/{p}} \, e^{-{t}/{p}}\, e^{- c \,|\xi_0-t|}\, \mathcal M_tf(\xi_t)\,dt.
$$

Now we apply H\"older's inequality, with $e^{- c \,|\xi_0-t|/p'}$ as one factor.
It follows that \eqref{mean} is not larger than constant times
$$
 \bigg(\int_{0}^{+\infty}\:e^{{\xi_0}} \, e^{-t}\,
 e^{- c \,|\xi_0-t|}\, \big[\mathcal M_tf(\xi_{t} )\big]^p\,dt\bigg)^{1/p}.
$$
Since this quantity is independent of the choice of the diamond $D$, it gives an upper bound for $M_{\mu}^{\mathcal{D}_+}f(\xi)$.

Integrating $p$th powers with respect to $d\mu(\xi)$, one obtains
$$
 \int \big[M_{\mu}^{\mathcal{D}_+}f(\xi)\big]^p \,d\mu(\xi)
 \lesssim  \int_{0}^{+\infty} \int_{\Pi_{t}}\, \int_{0}^{+\infty}\: e^{-t}\,
 e^{- c \,|\xi_0-t|}\, \big[\mathcal M_tf(\xi_{t})\big]^p\,dt\,d\lambda_t(\xi_t)\,d\xi_0.
$$

In the right-hand side here, we integrate first in $\xi_t$, using the fact$^{\ddag}$ that the operator $\mathcal M_t$
is bounded on $L^p(d\lambda_t)$ uniformly in $t$.
\footnotetext{$\ddag$ There are finitely many components of $\mathcal{M}_t$ defined by fixing the directions of the edges of
the parallelepipeds, and each of them is made by a linear transformation into
the  strong maximal operator $M_{\mathrm{str}}$ in $\mathbb{R}^{d-1}$, see \eqref{strong}.}
Thus the triple integral is at most constant times
$$
 \int_{0}^{+\infty} \int_{0}^{+\infty} \: e^{-t}\, e^{- c \,|\xi_0-t|}\,  \int_{\Pi_{t}} \,
 \,  f(\zeta )^p\,d\lambda_t(\zeta)\,dt\,d\xi_0.
$$
Integrating next in $\xi_0$, we conclude that
$$
 \int \big[M_{\mu}^{\mathcal{D}_+}f(\xi)\big]^p \,d\mu(\xi) \lesssim
 \int_{0}^{+\infty} \,e^{-t}\, \int_{\Pi_{t}}  \,
 \,  f(\zeta )^p\,d\lambda_t(\zeta)\,dt = \|f \|_{L^p(d\mu)}^p,
$$
and this proves the $L^p(d\mu)$-boundedness of $M_{\mu}^{\mathcal{D}_+}$.

\medskip

\begin{proof}[Proof of Proposition \ref{sophi}]
We fix $\xi\in D$ and $t \in (b, b+2r)$, and for convenience we also write
$t = b +h = z_0 - r +h$ with $0 < h < 2r$.  Further, we renumber the coordinates so that
\begin{equation}\label{renumber}
 z_d = \max_{1\le j \le d} z_j.
\end{equation}

Denote
 \begin{equation*}
  G_{t,\xi} = \Pi_t\cap \big\{x \in \mathbb{R}^d\colon \forall i\;\; x_i > - |t - \xi_{0}|\quad
     \mathrm{and}\quad |z-x|_1 < r + |t - \xi_{0}| \big\}.
 \end{equation*}
Obviously $D_t\subset G_{t,\xi}$ but also $\xi_t\in G_{t,\xi}$. Indeed, $(\xi_t)_i>(\xi_t)_i-\xi_i\ge-|\xi_t-\xi|_1=-|t-\xi_0|$ and
	$|z-\xi_t|_1\le |z-\xi|_1+  |\xi- \xi_{t}|_1  < r+|t - \xi_{0}|$.

In order to include $ G_{t,\xi}$ in a parallelepiped in $\Pi_t$, we let $x\in G_{t,\xi}$.
Since $|z-x|_1 < r + |t - \xi_{0}|$ and $(z-x)_0=r-h$, we then have for each $i=1,\ldots,d$
  \begin{equation}\label{3}
	x_i-z_i\le \sum_{1}^{d} (x_i-z_i)_+=\frac12\, \big[ |x-z|_1-(z-x)_0\big] <\frac12\, \big(h+|t - \xi_0|\big).
	\end{equation}

Switching coordinates to $y_i = z_i-x_i+(h+|t - \xi_0|)\slash2$, we get
 \begin{equation*}
 0 < y_i <  z_i + |t - \xi_{0}| + (h+|t - \xi_0|)\slash2 =   z_i + \big(h+3|t - \xi_0|\big)\slash2, \qquad i = 1,\dots, d.
  \end{equation*}
Further,  $(z-x)_0=r-h$ implies, since $y_d > 0$,
\begin{equation*}
 \sum_{1}^{d-1} y_i =  y_0 - y_d = (z-x)_0 + \frac{d(h+|t - \xi_0|)}2 -y_d < r-h +\frac{d(h+|t - \xi_0|)}2.
\end{equation*}

We need a simple lemma.

  \begin{lemma}\label{measure}
    Let $m \ge 2$ and consider the set  $E \subset \mathbb R^m$ defined by
    \begin{equation*}
   E = \Big\{y \in  \prod_1^m (0, a_i):\,
     \sum_{1}^{m}y_i < R \Big\}
    \end{equation*}
    for some $a_i,\: R > 0$. Then $E$ is contained in the $m$-dimensional rectangle
    \begin{equation*}
      \widetilde E = \prod_1^m (0, a_i\wedge R),
    \end{equation*}
    and the Lebesgue measures satisfy  $|E| \simeq \big|\widetilde E\big| = \prod_1^m a_i\wedge R$.
  \end{lemma}
\noindent (In expressions like the last product here, we always mean the product of the minima.)

To get the lower estimate for $|E|$ in the lemma, one observes that $E \supset \prod_1^m (0, (a_i\wedge R)\slash m)$,
and the other parts are trivial.
	
Let the projection $\pii_t\colon \Pi_{t} \to \mathbb R^{d-1}$ be given by suppression of the last coordinate.
The lemma, applied with $m = d-1$ and in the coordinates $(y_1,\dots,y_{d-1})$, implies that the projection $\pii_t(G_{t,\xi})$
is contained in a rectangle $\widetilde E$ in $\mathbb R^{d-1}$ with sides parallel to the $y$ (or equivalently $x$) coordinate axes.
Then $G_{t,\xi}$ is contained in  $\pii_t^{-1}\big(\widetilde E\big)$, which is seen to be a parallelepiped $P_{t}$
fulfilling the conditions of Proposition \ref{sophi},
except that the estimate we get for its Lebesgue measure is
  \begin{align} \nonumber
		\lambda_t(P_{t}) & \simeq  \prod_1^{d-1}  \bigg(z_i+\frac{h+3|t - \xi_0|}2\bigg) \wedge
		\bigg(r-h +\frac{d(h+|t - \xi_0|)}2\bigg) \\
		& \lesssim   \prod_1^{d-1} \big[z_i+ h + (t-b)\vee (\xi_0-b)\big] \wedge r. \label{paral}
  \end{align}

In addition to \eqref{paral}, we will deduce a similar estimate by writing first
  \begin{equation*}
  \sum_{1}^{d-1} x_i = z_0 - (z - x)_0 -x_d <  z_0 - r + h + |t - \xi_0|.
  \end{equation*}
Combining this estimate with \eqref{3} and applying Lemma \ref{measure} in the coordinates $y_i = x_i + |t - \xi_0|$,
we can argue as above. As a result, we find a parallelepiped $P_{t}$ containing $G_{t,\xi}$ and verifying
  \begin{equation}\label{paral.addit}
		\lambda_t(P_{t})  \lesssim \prod_1^{d-1} \big[z_i+ h + (t-b)\vee (\xi_0-b)\big] \wedge \big(z_0 - r +h+ |t - \xi_0|\big).
  \end{equation}

Next we derive two different lower estimates for $\mu(D)$, whose validity will depend on the condition
\begin{equation}\label{additional}
    \sum_{1}^{d-1} z_i \ge r-h.
  \end{equation}
  We shall verify that
 \begin{equation}\label{estD}
      \mu(D) \gtrsim e^{-b}\, \prod_1^{d-1} (z_i+1) \wedge \bigg\{\begin{matrix}r\\ z_0-r\end{matrix}\bigg\}
    \end{equation}
when \eqref{additional} holds (upper), and when \eqref{additional} is false (lower), respectively.

These two estimates will end the proof of Proposition \ref{sophi}
when combined with \eqref{paral} and \eqref{paral.addit}, respectively, since
  \begin{equation*}
	\frac {z_i+ h + (t-b)\vee (\xi_0-b)} {z_i+1} \lesssim 1+ (t-b)\vee (\xi_0-b),
  \end{equation*}
and
  \begin{equation*}
  \frac {[z_i+ h + (t-b)\vee (\xi_0-b)] \wedge (z_0 - r +h+ |t - \xi_0|)} {(z_i+1) \wedge (z_0 - r)} \lesssim 1+ (t-b)\vee (\xi_0-b);
  \end{equation*}
recall here that $z_0 - r > 2$, see Reduction 2.

To verify \eqref{estD}, it is enough to show that for $1<h<2$, under relevant assumptions,
    \begin{equation}\label{estPy0}
      \lambda_{b+h}(D_{b+h}) \gtrsim  \prod_1^{d-1} (z_i+1) \wedge \bigg\{\begin{matrix}r\\ z_0-r\end{matrix}\bigg\},
    \end{equation}
because one can then integrate with respect to $e^{-b-h}\,dh$ over the interval $(1,2)$.
	
Observe that the last coordinate of any point $x \in  \Pi_{b+h}$ is given by
 \begin{equation}\label{xd}
	x_d = z_d +  \sum_{1}^{d-1} (z_i-x_i) - r+ h  = z_0 -  \sum_{1}^{d-1} x_i - r+ h.
 \end{equation}
		
Let $h \in (1, 2)$. Aiming at the lower case  in \eqref{estPy0} and thus assuming  \eqref{additional} false, we define the set
\begin{equation*}
   E =
    \bigg\{(x_i)_1^{d-1} \in \mathbb{R}^{d-1}: \: 0 < x_i < z_i + \frac h {2d},\; \; i = 1,\dots, d-1,
 \; \;\, \mathrm{and}\;\;\; \sum_{1}^{d-1} x_i  <  z_0 - r  \bigg\}.
\end{equation*}
 We claim that the  inverse projection, or lift,
 $\pii_{b+h}^{-1}(E)$ is contained in $D_{b+h}$. Indeed, let $x \in\pii_{b+h}^{-1}(E)$
 and consider the last coordinate $x_d$ of $x$. From \eqref{xd} we conclude
  $$
   x_d   > z_0 -(z_0 - r) -r + h = h \qquad \mathrm{and} \qquad  x_d   < z_0  - r+ h < z_d,
  $$
  the last step since \eqref{additional} is false. Thus $x \in \mathbb{R}^d_+ \cap  \Pi_{b+h}$. Further,
\begin{equation}\label{distance}
     |x - z|_1  =   (z - x)_0 + 2\sum_{1}^{d} (x_i - z_i)_+   < r-h + \frac{2(d-1)h}{2d} < r,
\end{equation}
so that $x \in D$. The claim follows.

For the measures, we then get $\lambda_{b+h}(D_{b+h}) \ge \lambda_{b+h}(\pii_{b+h}^{-1}(E)) \simeq  |E|$, where
$|\cdot|$ denotes Lebesgue measure in $\mathbb R^{d-1}$.
Lemma \ref{measure} yields that $|E| \simeq   \prod_1^{d-1} (z_i+1) \wedge (z_0 - r)$.
This proves \eqref{estPy0} and thus \eqref{estD}, for the lower lines.
	
Next, we verify  \eqref{estPy0} (upper), under the assumption \eqref{additional}; recall that $1 < h <2$.
We start with the case $z_d \ge r$, and here we argue almost as above. Define now
 \begin{equation*}
   E' =
    \bigg\{(x_i)_1^{d-1} \in \mathbb{R}^{d-1}: \: 0 < x_i < z_i +\frac h {2d}, \; i = 1,\dots, d-1,
  \,\; \mathrm{and}\;\, \sum_{1}^{d-1} x_i  >   \sum_{1}^{d-1} z_i - r + h \bigg\}.
    \end{equation*}
     Points in $E'$ clearly satisfy
$$
 -\frac{(d-1)h}{2d} < \sum_{1}^{d-1} (z_i-x_i)  < r-h.
$$
As before, we take a point $x \in \pii_{b+h}^{-1}(E')$
    and verify that  $x \in D_{b+h}$.    From    \eqref{xd} combined with  $z_d \ge r$, we now get
$$
   x_d   > z_d -(d-1)h/(2d) -r + h > 0 \qquad \mathrm{and} \qquad  x_d   < z_d + r - h - r+ h = z_d.
$$
It follows that $x \in \mathbb{R}^d_+ \cap  \Pi_{b+h}$ and that \eqref{distance} remains valid.
This proves the inclusion $\pii_{b+h}^{-1}(E')  \subset D_{b+h}$.

Thus $\lambda_{b+h}(D_{b+h}) \gtrsim  |E'|$, and $|E'|$
can be estimated by means of Lemma \ref{measure} and the coordinates $y_i =  z_i - x_i + h/(2d),\; i=1,\dots, d-1$.
Since $0 < y_i < z_i+h/(2d)\simeq z_i + 1$ for each $i$ and $\sum_{1}^{d-1} y_i < r-h + (d-1)h/(2d) \simeq r$,
the result is $|E'| \simeq \prod_1^{d-1} (z_i+1) \wedge r$. This proves \eqref{estPy0} (upper) when $z_d \ge r$.

In the complementary case $z_d<r$, we can suppress $\wedge r$ in \eqref{estPy0} (upper) because of \eqref{renumber}.
Define $s,\sigma \in \mathbb{R}$ by
 \begin{equation*}
  s\,  \sum_{1}^{d-1} \Big( z_i+\frac h {2d} \Big)  = \sum_{1}^{d-1} z_i - r+h, \qquad
 \sigma\,  \sum_{1}^{d-1} \Big( z_i+ \frac h {2d}\Big) = z_0 - r.
 \end{equation*}
They satisfy $0 \le s < \sigma < 1$, where the first inequality follows from \eqref{additional},
the second because $h<2<z_d$ and the third from $z_d < r$.
Consider now the set
\begin{equation*}
   S =
    \bigg\{x \in \Pi_{b+h}\colon s \Big(  z_i + \frac h {2d}\Big) < x_i < \sigma  \Big( z_i +\frac h {2d}  \Big),\;
			\;i=1,\dots, d-1  \bigg\}.
\end{equation*}

Clearly, any point $x \in S$ satisfies
\begin{equation*}
  \sum_{1}^{d-1} z_i - r+h  < \sum_{1}^{d-1} x_i  <   z_0 - r,
\end{equation*}
so for its last coordinate, \eqref{xd} implies $0<h<x_d<z_d$. Thus $S \subset \mathbb{R}^d_+$, and
\eqref{distance} holds again, since for each $i=1,\dots, d-1$
 \begin{equation*}
 x_i - z_i < (\sigma-1) z_i +\sigma h/(2d) < h/(2d).
 \end{equation*}
It follows that $S \subset D_{b+h}$.

For the measures, we have
 \begin{equation*}
 \lambda_{b+h}(S)  \simeq \big|\pii_{b+h}(S)\big| =  \prod_{1}^{d-1} (\sigma-s)
	\Big(z_i + \frac h{2d}\Big) \simeq  (\sigma-s)^{d-1}\,\prod_{1}^{d-1} (z_i + 1).
 \end{equation*}
To finish the proof of   \eqref{estPy0} (upper), it is enough to verify that $\sigma-s \gtrsim 1$. But
 \begin{equation*}
 \sigma-s= \frac{z_d-h}{\sum_{1}^{d-1} [z_i+ h/(2d)]} \gtrsim \frac{z_d}{\sum_{1}^{d-1} z_i + 1}  \gtrsim 1,
 \end{equation*}
the last inequality because of \eqref{renumber}. Proposition \ref{sophi} is proved.
\end{proof}

\subsection*{Proof of Theorem \ref{thm:main}(B), the case of $\mathbf{M_{\mu}^{\mathcal{B}}}$}
Our strategy of proving the $L^p$-boundedness of ${M_{\mu}^{\mathcal{B}_+}}$
is heavily inspired by \cite{LiSj}. Thus we first
rotate suitably the whole situation and then use a slicing argument together with $L^p$-boundedness
of certain standard maximal functions. The details are as follows.

Rotate simultaneously the cone $\R$ and all the objects considered (measure, truncated balls, etc.)
so that the rotation of $\Sigma_+^{d-1}$ is orthogonal to the first coordinate axis and contained in the half-space
$\{x \in \mathbb{R}^d : x_1 > 0\}$.
Then denote by $C_+$ the rotated open cone, in which the rotated measure is, up to a multiplicative constant and scaling,
$$
{\nu}(dx) = e^{-x_1}\,dx.
$$
Clearly, the above formula extends $\nu$ from $C_+$ to all of $\mathbb{R}^d$. We shall sometimes use this extension
without explicit indication. Further, denote
$$
\pi_{\xi} = \{x \in \mathbb{R}^d: x_1 = \xi\}, \qquad \xi \ge 0.
$$

Our aim is to prove that $M_{\mu}^{\mathcal{B}_+}$ is bounded on $L^p(\mathbb{R}^d_+,d\mu)$ for $1<p<\infty$.
After rotation and scaling and keeping the same symbols, we consider $M_{\mu}^{\mathcal{B}_+}$
as a maximal operator acting on functions living on $C_+$, related to the family $\widetilde{\mathcal{B}}_+$
of truncated Euclidean balls in $\mathbb{R}^d$ with centers in $C_+$,
the truncation being relative to $C_+$. Then the $L^p$-boundedness concerns $L^p(C_+,d\nu)$.

Thus it is enough that we prove the $L^p(d\nu)$-boundedness, $1<p<\infty$, of the maximal operator
\begin{equation} \label{max_B}
Mf(x) = \sup \frac{1}{\nu(\widetilde{B})} \int_{\widetilde{B}} |f|\, d\nu,
\end{equation}
where the supremum is taken over all truncated balls
$$
\widetilde{B} = \widetilde{B}(m,r) := \mathbf{B}(m,r) \cap C_+,
$$
called simply balls henceforth, such that $m \in C_+$ and $x \in \widetilde{B}$.
Further, we may assume that $f$ is non-negative and defined in all of $\mathbb{R}^d$ but supported in the closure of $C_+$.

In what follows points in $\mathbb{R}^d$ will be written as $x=(x_1,x')$.
We shall always assume that the centers of balls $\widetilde{B}$ are in $C_+$.
Given $\widetilde{B}$, the minimum
$$
\min\big\{x_1 : x \in \cl(\widetilde{B})\big\}
$$
($\cl$ meaning closure in $\mathbb{R}^d$) is taken at a unique point $a=a(\widetilde{B})=(a_1,a') \in \partial \widetilde{B}$.

We now make some preliminary observations that lead to an essential reduction of the class of truncated balls over which
the supremum in \eqref{max_B} is taken.

\noindent \textbf{Observation 1.}
We may restrict to balls $\widetilde{B}(m,r)$ with radii uniformly bounded from below by a positive constant,
see Reduction 1 above. In addition we may assume that $a_1(\widetilde{B}) > 2$, see Reduction 2.

\noindent \textbf{Observation 2.}
We may further restrict to balls $\widetilde{B}(m,r)$ such that $a=a(\widetilde{B}) \in \partial C_+$.
(In particular, we exclude untruncated balls $\widetilde{B} = \mathbf{B}(m,r)$ entirely contained in $C_+$.)
Indeed, if $a \notin \partial C_+$, i.e., $a$ is in (the interior of) $C_+$,
then $m_1 = a_1 + r$ and $m'=a'$, and one considers the following two complementary cases.

If $1 \lesssim \sqrt{r} < a_1$, then $\nu(\widetilde{B}(m,r)) \simeq r^{(d-1)/2}e^{-a_1} \simeq \nu(\mathbf{B}(m,r))$
(for the last relation, see the proof of Lemma \ref{lem:balls}) and the result is a simple consequence of \cite[Theorem 3]{LiSj}.

On the other hand, letting $M_0$ be that part of the maximal operator $M$ given by restricting the supremum in \eqref{max_B}
to balls $\widetilde{B}(m,r)$ remaining after Observation 1 and such that $a(\widetilde{B}) \notin \partial C_+$ and
$a_1 \le \sqrt{r}$, we have the following.

\noindent \textbf{Claim:} $M_0$ is $L^p(d\nu)$-bounded for $1 < p < \infty$.

To justify the Claim, notice that any $\widetilde{B}$ under consideration contains a cylinder parallel to the $x_1$ axis,
with one face contained in $\pi_{a_1 + 1}$,
of essentially unit width and radius comparable to $a_1$, so $\nu(\widetilde{B}) \gtrsim a_1^{d-1}e^{-a_1}$.
Given that, consider the projections
\begin{align*}
d\tau(x_1) & = x_1^{d-1} e^{-x_1}\, dx_1, \\
F(x_1) & = \frac{1}{x_1^{d-1}} \int_{\pi_{x_1} \cap C_+} f(x_1,x')\, dx',
\end{align*}
of $d\nu$ and $f$, respectively, on the $x_1$ axis (here we omit multiplicative constants,
which are irrelevant for the argument). Notice that $\int f d\nu = \int F d\tau$.
Thus we have
$$
\frac{1}{\nu(\widetilde{B})} \int_{\widetilde{B}} f \, d\nu \lesssim \frac{1}{\tau(I_{a_1})} \int_{I_{a_1}} F\, d\tau,
$$
where $I_{a_1} = (a_1,\infty)$. Now observe that the one-dimensional maximal operator
$$
g(s) \mapsto \sup_{I \ni s} \frac{1}{\tau(I)} \int_I g \, d\tau
$$
(the supremum taken over all intervals $I \subset \mathbb{R}_+$ such that $s \in I$) is of weak type
$(1,1)$ with respect to the measure space $(\mathbb{R}_+,d\tau)$, and it controls $M_0$.
Therefore $M_0$ is of weak type $(1,1)$ with respect to $(C_+,d\nu)$, and the $L^p(d\nu)$-boundedness of $M_0$
follows by interpolation with the $L^{\infty}$-boundedness. This finishes proving the Claim and ends Observation 2.

Summing up, in the analysis of \eqref{max_B} we may assume that $\widetilde{B}=\widetilde{B}(m,r)$ is a ball such that
$m \in C_+$ and
\begin{equation} \label{obs12}
r>\sqrt{d}, \qquad a_1 > 2, \qquad a \in \partial C_+.
\end{equation}
By convention, we define the supremum in \eqref{max_B} as zero if there is no admissible ball $\widetilde{B}$ containing $x$.

We shall first prove the result in the simplest situation when the dimension $d=2$. This will give us some intuition needed
for higher dimensions.

\bigskip

\textbf{Dimension $\boldsymbol{d=2}$.}
When $d=2$ we write points simply $x=(x_1,x_2)$ rather than $x=(x_1,x')$.
Our rotated cone is $C_+ = \{ x \in \mathbb{R}^2 : |x_2| < x_1\}$.
We can assume that the balls $\widetilde{B}(m,r)$ under consideration are such that $m_2 \ge 0$, by symmetry.
Then $a(\widetilde{B}) = (\A,\A)$ with $\A > 2$, and also $r > \sqrt{2}$ and $m_2 \ge \A$; see \eqref{obs12}.
Notice that $r/\sqrt{2} < m_1 - \A \le r$ and, of course,
\begin{equation} \label{circ}
(m_1 - \A)^2 + (m_2 - \A)^2 = r^2.
\end{equation}
See Figure \ref{fig25}.
\begin{figure}[ht]
\includegraphics[height=0.5\textwidth]{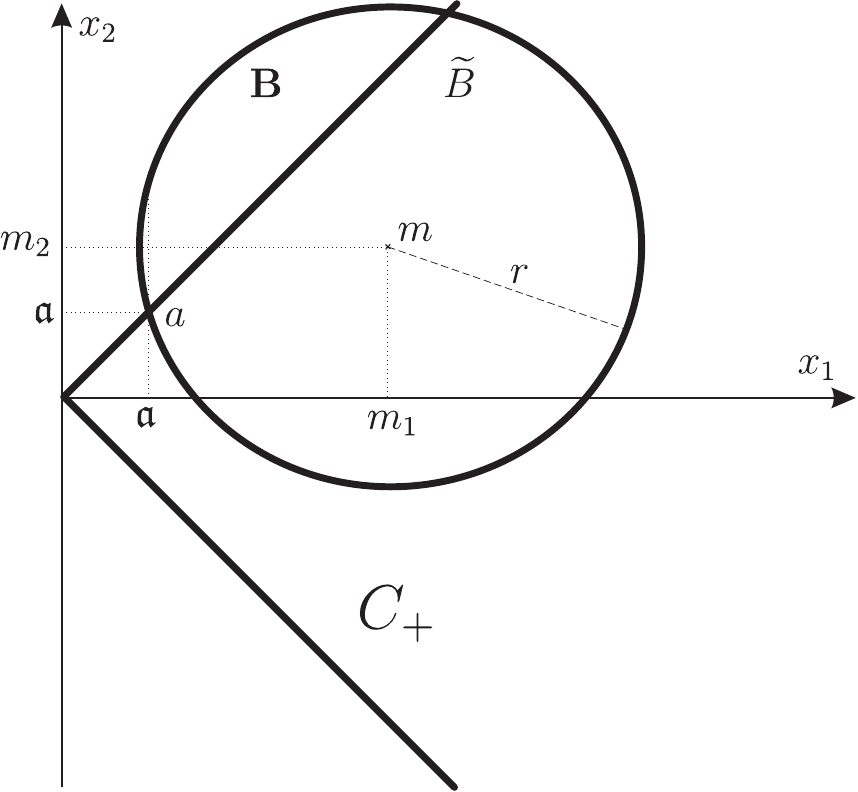}
\caption{The  situation for $d=2$.}
\label{fig25}
\end{figure}

We shall now split into cases.
In each case, we consider the maximal operator obtained by imposing some conditions on $\widetilde{B}$,
in addition to \eqref{obs12}.

\noindent \textbf{Case 1:} $\widetilde{B}$ contains the point $(\A+1,0)$.
Then $\nu(\widetilde{B}) \gtrsim \A e^{-\A}$, since $\widetilde{B}$ contains a rectangle of unit width and
height $\A$, with one of the vertical edges contained in $\pi_{\A+1}$. Thus the projection argument from Observation 2 gives
the desired conclusion.

\noindent \textbf{Case 2:} $\widetilde{B}$ does not contain the point $(\A+1,0)$.
We first find the lower intersection of the line $x_1 = \A + h$, $0 < h \le r/\sqrt{2}$, with $\partial \mathbf{B}$,
denoted $(\A+h, \A-\xi)$; here $\mathbf{B}$ is the untruncated prototype of $\widetilde{B}$ and  $\xi = \xi(h) > 0$.
Notice that the condition defining Case 2 can be written as $\xi(1) \le \A$.

We have
$$
(m_1 - \A - h)^2 + (m_2 - \A + \xi)^2 = r^2.
$$
Subtracting \eqref{circ} from this equation leads to
$$
\xi^2 + 2(m_2 - \A)\xi - 2(m_1 - \A)h + h^2 = 0.
$$
Dividing by $\xi^2$, solving for $1/\xi$ and taking into account that $\xi > 0$, we get
$$
\xi = \frac{2(m_1 - \A)h - h^2}{m_2 - \A + \sqrt{(m_2-\A)^2+2(m_1-\A)h - h^2}}.
$$
Note that $h < m_1 - \A$ (recall that $r/\sqrt{2} \le m_1-\A < r$). Then
$2(m_1-\A)h - h^2 \simeq (m_1-\A)h \simeq rh$. Consequently,
\begin{equation} \label{xiest}
\xi(h) \simeq \frac{rh}{m_2-\A + \sqrt{rh}} \simeq \frac{rh}{m_2-\A} \wedge \sqrt{rh}, \qquad 0 < h \le r/ \sqrt{2}.
\end{equation}
To estimate $\nu(\widetilde{B})$ from below, observe that $\widetilde{B}$
contains the triangle $T$ whose vertices are $(\mathfrak{a}, \mathfrak{a})$,  
$(\mathfrak{a}+1, \mathfrak{a}-\xi(1))$  and $(\mathfrak{a}+1, \mathfrak{a})$, and
$\nu(T)\simeq \xi(1) e^{-\A}$. Thus \eqref{xiest} implies
\begin{equation} \label{nest}
\nu(\widetilde{B}) \gtrsim \xi(1) e^{-\A} \simeq \Big( \frac{r}{m_2-\A} \wedge \sqrt{r}\Big) e^{-\A}.
\end{equation}

Next, we consider all $h > 0$ and estimate from above the measures of the intersections $\pi_{\A +h} \cap \shadow \widetilde{B}$,
where
$$
\shadow \widetilde{B} := \widetilde{B} + \big( \mathbb{R}_+ \times \{0\}\big)
$$
is the shadow of $\widetilde{B}$ in the positive $x_1$ direction.
By the geometry of the situation
and \eqref{xiest}, observing also that $m_2 - \A < r/\sqrt{2}$, we have
$$
\big| \pi_{\A+h} \cap \shadow\widetilde{B}\big| \le
		\begin{cases}
			\xi(h) + h, & h \le r/ \sqrt{2} \\
			2r, & h > r/ \sqrt{2}
		\end{cases} \Bigg\}
\lesssim
	\begin{cases}
		\frac{rh}{m_2-\A} \wedge \sqrt{rh}, & h \le r/ \sqrt{2} \\
		r, & h > r/ \sqrt{2}
	\end{cases} \Bigg\},
\qquad h > 0.
$$
Using this together with \eqref{nest}, by an elementary analysis of cases we see that
\begin{equation} \label{ratio_est}
\frac{|\pi_{\A+h}\cap \shadow\widetilde{B}|}{e^{\A}\nu(\widetilde{B})} \lesssim
	\begin{cases}
		\sqrt{h} + h, & h \le r/ \sqrt{2} \\
		\sqrt{r} + r, & h > r/ \sqrt{2}
	\end{cases} \Bigg\}
\lesssim \sqrt{h} + h \lesssim 1 + h, \qquad h > 0,
\end{equation}
uniformly in $\A$ and $\widetilde{B}$.

Now, let $M_2$ be the part of the maximal operator \eqref{max_B} under consideration, i.e., with the supremum taken
only over balls $\widetilde{B}$ considered in Case 2. We will apply the slicing argument from \cite{LiSj}.

Similarly as in \cite{LiSj}, consider the unit slices
$$
S_i = \{x \in C_+ : i < x_1 \le i+1 \}, \qquad i \ge 1.
$$
In $S_i$ one has $e^{-i-1}dx \le d\nu(x) \le e^{-i}dx$. Let
$$
M_2^k f(x) = \sum_{j-i=k} \chi_{S_j}(x) M_2(f\chi_{S_i})(x).
$$
Since $M_2 f \le \sum_{k \in \mathbb{Z}} M_2^kf$, it is enough to prove that
$\|M_2^kf\|_{L^p(d\nu)} \lesssim 2^{-\delta|k|/p} \|f\|_{L^p(d\nu)}$ with some $\delta > 0$, because
then one can sum the estimates and get the conclusion.
Thus we must show that
\begin{equation} \label{nkey_est}
\int_{S_j} \big[M_2(f\chi_{S_i})\big]^p\, d\nu \lesssim e^{-\delta |j-i|} \int_{S_i} f^p \, d\nu,
	\qquad i,j \ge 1.
\end{equation}

With $i,j \ge 1$, we let $x \in S_j$ and $\widetilde{B}$ be a ball containing $x$, and
we will estimate first the mean
$$
\frac{1}{\nu(\widetilde{B})} \int_{\widetilde{B}} \chi_{S_i}(y) f(y)\, d\nu(y).
$$
In our situation $x \in \widetilde{B} \cap S_j$ and
$y \in \widetilde{B} \cap S_i$.
Observing that the sets $\{z_2 \in \mathbb{R}: \exists z_1 \; (z_1,z_2) \in \pi_{\A+h} \cap \shadow \widetilde{B}\}$
form an increasing family of intervals with respect to $h > 0$, we get
$$
|y_2-x_2| \le \big|\pi_{i+1} \cap \shadow\widetilde{B}\big| \vee \big|\pi_{j+1}\cap\shadow\widetilde{B}\big|
	= \big|\pi_{i\vee j + 1} \cap \shadow\widetilde{B}\big|;
$$
notice that here $i,j \ge 1 \vee (\A-1) = \A - 1$.
Then, using \eqref{ratio_est}, we obtain
\begin{align} \nonumber
\frac{1}{\nu(\widetilde{B})} \int_{\widetilde{B}} \chi_{S_i}(y) f(y)\, d\nu(y) &
	\quad \le
	\frac{1}{\nu(\widetilde{B})}
	\int_i^{i+1} e^{-i} \int_{|y_2-x_2| < |{\pi_{i\vee j + 1}} \cap\, \shadow\widetilde{B}|}
		f(y_1,y_2)\, dy_2\, dy_1 \\ \nonumber
& \quad \le 2 \,
	\frac{|{\pi_{i\vee j + 1}} \cap \shadow\widetilde{B}|}{\nu(\widetilde{B})}
	\int_i^{i+1} e^{-i} \int_i^{i+1} \mathcal{M}f(y_1,x_2)\, dy_1 \\
& \quad \lesssim \big[ 1 + (i \vee j - \A +1)\big] e^{\A-i} \int_i^{i+1} \mathcal{M}f(y_1,x_2)\, dy_1, \label{n420}
\end{align}
where the implicit multiplicative constant is independent of $i,j \ge \mathfrak{a}-1$, the ball $\widetilde{B}$ and the point
$x \in \widetilde{B} \cap S_j$, and of $f$.
Here $\mathcal{M}$ is the one-dimensional non-centered Hardy-Littlewood maximal function
acting on the second coordinate. Note that $\mathcal{M}$ is bounded on $L^p(\mathbb{R},dx_2)$ for $p > 1$.

We now estimate the factor in front of the integral in \eqref{n420}. Write
$$
\A-i-1 = \frac{j-i}{p} - \frac{1}{p'}(i-\A+1)-\frac{1}p (j-\A+1) \le \frac{j-i}{p}
	- \Big( \frac{1}{p} \wedge \frac{1}{p'}\Big) [i \vee j- \A +1],
$$
where the last inequality follows from the bound $i \wedge j \ge \A-1$.
Thus
\begin{align*}
\big[ 1 + ( i\vee j - \A + 1)\big] e^{\A-i}
& \le e^{j/p- i/p +1} e^{-2\epsilon [i\vee j - \A +1]} \big[1+(i\vee j - \A +1)] \\
& \lesssim e^{j/p- i/p} e^{-\epsilon [i\vee j - \A +1]}\\
& \le e^{j/p- i/p} e^{-\epsilon |i-j|},
\end{align*}
with $\epsilon = \frac{1}{2}(\frac{1}p \wedge \frac{1}{p'})$, uniformly in $\A>2$ and $i,j \ge \A-1$.

With the bound just obtained, taking the supremum of the left-hand side of \eqref{n420} and using H\"older's inequality
on the right-hand side there, we arrive at
$$
\chi_{S_j}(x) M_2(f\chi_{S_i})(x) \lesssim e^{j/p- i/p} e^{-\epsilon |i-j|} \chi_{S_j}(x)
	\bigg(\int_{i < y_1 < i+1} \big[\mathcal{M}f(y_1,x_2)\big]^p \, dy_1\bigg)^{1/p}.
$$
Raising to power $p$ and integrating this estimate in $x=(x_1,x_2) \in S_j$ we further get
$$
\int_{S_j} \big[M_2(f\chi_{S_i})(x)\big]^p \,e^{-j} dx \lesssim e^{-\epsilon p |j-i|}
	\int_{i < y_1 < i+1} \int_{S_j} \big[\mathcal{M}f(y_1,x_2)\big]^p \, dx_1\,  dx_2 \, e^{-i}\,  dy_1.
$$
Finally, we use the $L^p$-boundedness of $\mathcal{M}$ to write
$$
\int_{\mathbb{R}} \big[\mathcal{M}f(y_1,x_2)\big]^p \, dx_2 \lesssim \int_{\mathbb{R}} f^p(y_1,y_2)\, dy_2,
\qquad y_1 \in (i,i+1),
$$
and \eqref{nkey_est} with $\delta = \epsilon p$ follows. This finishes the proof in the case of dimension $d=2$.

\begin{rem*}
Cases 1 and 2 considered above can be merged. Indeed, right after \eqref{circ} we can estimate $\xi(h)$, as it was done
in Case 2, getting \eqref{xiest}. Then it follows that
$$
\nu(\widetilde{B}) \gtrsim \big[ \A \wedge \xi(1) \big] e^{-\A} \simeq
	\Big( \A \wedge \frac{r}{m_2-\A} \wedge \sqrt{r}\Big) e^{-\A}.
$$
Further, we can estimate measures of the intersections $\pi_{\A+h} \cap \shadow \widetilde{B}$ as
(observe that the expression $2(\A + h)$ appears as the measure of $C_{+} \cap \pi_{\A+h}$)
\begin{align*}
\big| \pi_{\A+h} \cap \shadow\widetilde{B}\big| & \le
		\begin{cases}
			2(\A+h)\wedge(\xi(h) + h), & h \le r/ \sqrt{2} \\
			2r, & h > r/ \sqrt{2}
		\end{cases} \Bigg\} \\
& \lesssim
	\begin{cases}
		(\A+h) \wedge \frac{rh}{m_2-\A} \wedge \sqrt{rh}, & h \le r/ \sqrt{2} \\
		r, & h > r/ \sqrt{2}
	\end{cases} \Bigg\},
\qquad h > 0.
\end{align*}
Using this together with an elementary analysis of cases we get the key bound
$$
\frac{|\pi_{\A+h}\cap \shadow\widetilde{B}|}{e^{\A}\nu(\widetilde{B})} \lesssim
	1 + h, \qquad h > 0,
$$
uniformly in $\A$ and $\widetilde{B}$.
From here the slicing argument goes as described in Case 2 above.
\end{rem*}

\bigskip

\textbf{Dimension $\boldsymbol{d=3}$.}
From now on we will write points $x=(x_1,x')$, with $x_1 > 0$ and $x' \in \mathbb{R}^2$.
Our fixed rotated cone $C_+$ is contained in $\mathbb{R}_+ \times \mathbb{R}^2$, its vertex is the origin of $\mathbb{R}^3$,
and its central axis is the $Ox_1$ axis.
For any $\xi>0$, the intersection $C_+ \cap \pi_{\xi}$ is an open equilateral triangle of side $\sqrt{6}\,\xi$.

In order to estimate $Mf$ defined in \eqref{max_B}, we let $\widetilde{B} = \widetilde{B}(m,r)$ be a truncated
ball with $m \in C_+$ verifying \eqref{obs12}.

For any set $E \subset \mathbb{R}^3$, we define its shadow in the direction of the $x_1$ axis as
$$
\shadow E := E +  \{(s,0,0) : s > 0\}.
$$

We claim that 
\begin{equation}\label{m1-a1}
1 < r\slash \sqrt{3} \le m_1 - a_1 \le r,
\end{equation}
where only the second inequality needs to be verified. For this we fix $m_1$ and $a_1$ and use Figure \ref{fig99}.
\begin{figure}[ht]
\includegraphics[height=0.4\textwidth]{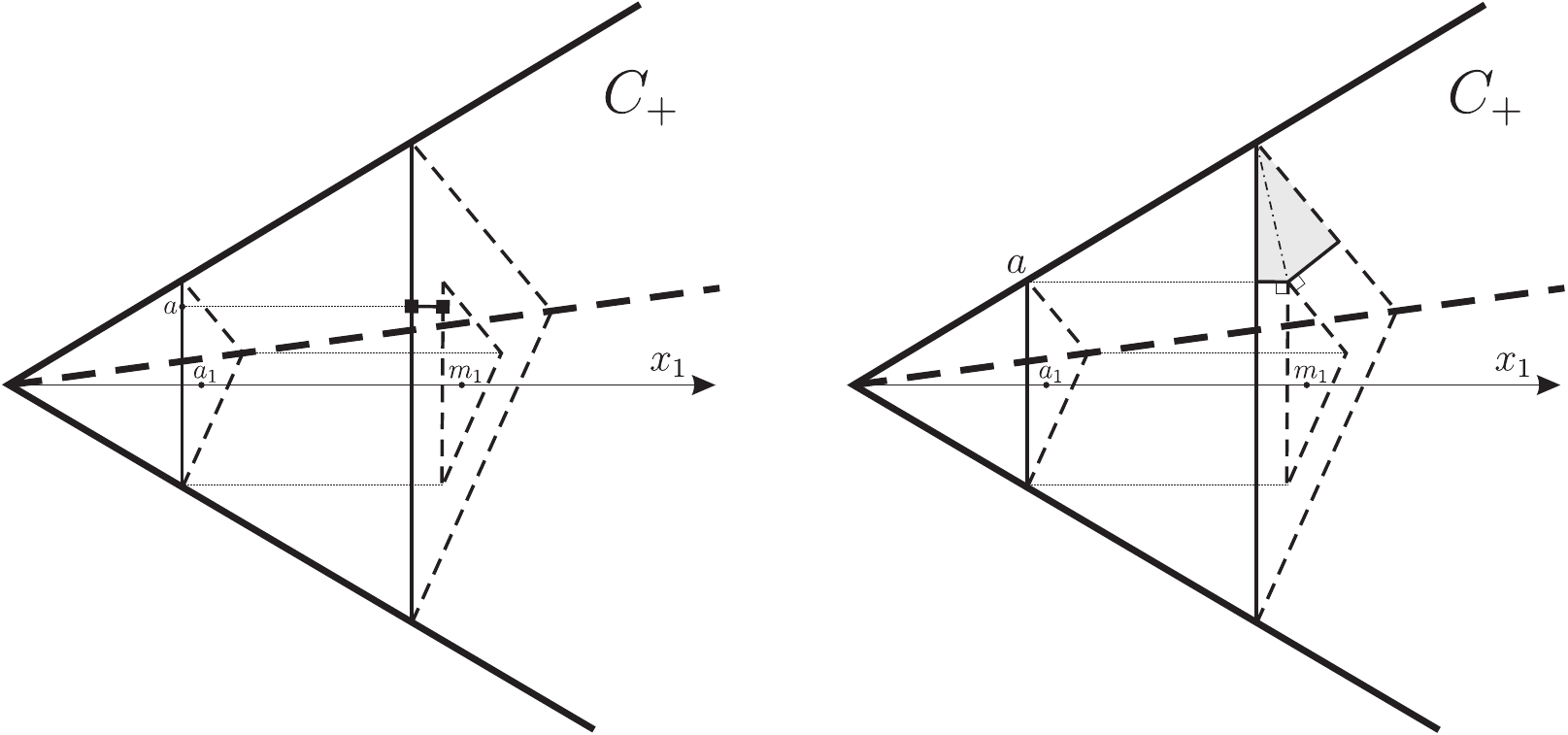}
\caption{Positions of $a$ and $m$, $d=3$.}
\label{fig99}
\end{figure}
Each part of this figure shows the triangles $\pi_{a_1} \cap \partial C_+$ and $\pi_{m_1} \cap \partial C_+$,
and inside the latter the triangle $\pi_{m_1} \cap \shadow(\pi_{a_1} \cap \partial C_+)$.
Notice that the point $m$ cannot be in the interior of this last triangle, since $a$ is on the
boundary of $C_+$. Given the position of $a$, the figure illustrates the possible positions of $m$.
To the left, $a$ is on an open face of the cone $C_+$, and then $m$ is seen to be in the short,
closed segment indicated. In the right-hand part of the figure, $a$ is on an edge of $C_+$, and $m$
has to belong to the closed quadrilateral marked in the figure. From this, we see that the
minimal value of the quotient $(m_1 - a_1)\slash r = (m_1 - a_1)\slash |m - a|$ occurs when $a$ and $m$ are
situated on the same edge of $C_+$, and then the quotient equals $1\slash \sqrt{3}$. We have verified the claim  \eqref{m1-a1}.

For $0 \le h < r+(m_1-a_1)$ we define
$$
C_h = \big\{ x' \in \mathbb{R}^2 : (a_1+h,x') \in C_+\big\} \qquad \textrm{and} \qquad
B_h = \big\{ x' \in \mathbb{R}^2 : (a_1+h,x') \in \mathbf{B} \big\},
$$
with $\mathbf{B}= \mathbf{B}(m,r)$. Observe that $B_h$ would be empty if defined in this way for $h \ge r+(m_1-a_1)$.

\medskip

\noindent \textbf{Case I:} $a(\widetilde{B})$ lies on an edge of $C_+$.

We intersect $C_+$ and $\mathbf{B}$ with $\pi_{a_1}$, see Figure \ref{fig3r}.
\begin{figure}[ht]
\includegraphics[height=0.5\textwidth]{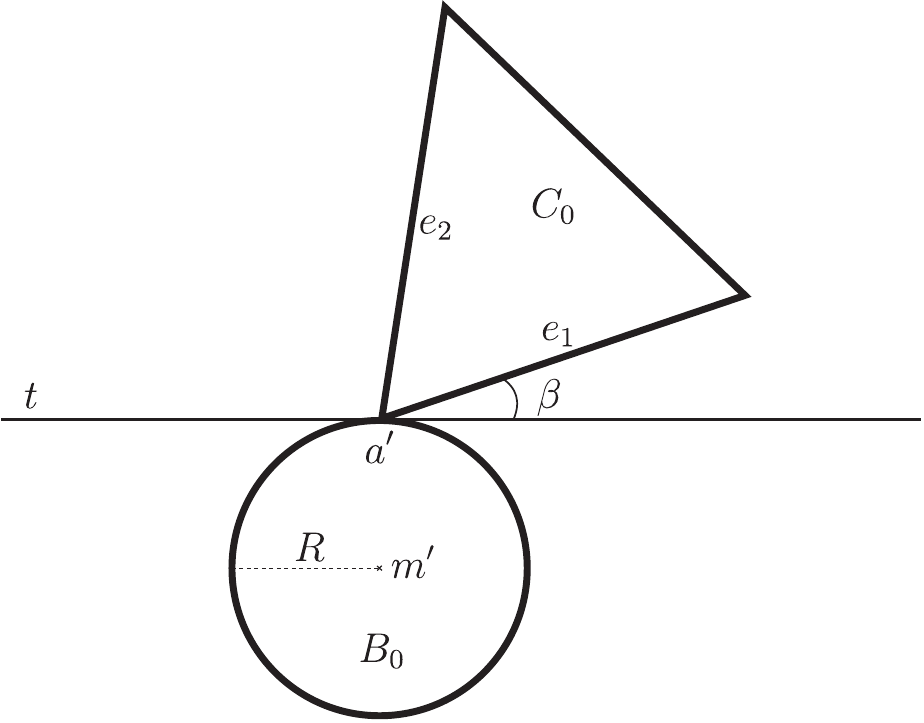}
\caption{The plane $\pi_{a_1}$, $d=3$.}
\label{fig3r}
\end{figure}
Then $C_0$ is an equilateral triangle with one vertex at $a'$, and $B_0$ is an open disc with center $m'$ and radius $R$ satisfying
\begin{equation} \label{eq:p1r}
(m_1-a_1)^2 + R^2 = r^2.
\end{equation}
The definition of $a$ implies that $a' \in \partial B_0 \cap \partial C_0$
and also that the tangent line, denoted $t$, to $B_0$ through $a'$ does not intersect $C_0$. Thus $R = |a'-m'|$.

The point $a'$ is the endpoint of two edges of $C_0$, and we consider the angles at $a'$ between $t$ and these two edges.
Let $\beta$ denote the smallest such angle and let $e_1$ be the corresponding edge. Then $0 \le \beta \le \pi/3$, and
the other edge $e_2$ forms the angle of $\beta + \pi/3$ with the same tangent.

We now consider the intersection of $\mathbf{B}$ and $\shadow(C_+ \cap \pi_{a_1})$ with
the plane $\pi_{a_1+h}$, assuming that $0 < h \le r/\sqrt{3}$; see Figure \ref{fig4r}. Then $a'$ is an inner point of the disc $B_h$.
\begin{figure}[ht]
\includegraphics[height=0.5\textwidth]{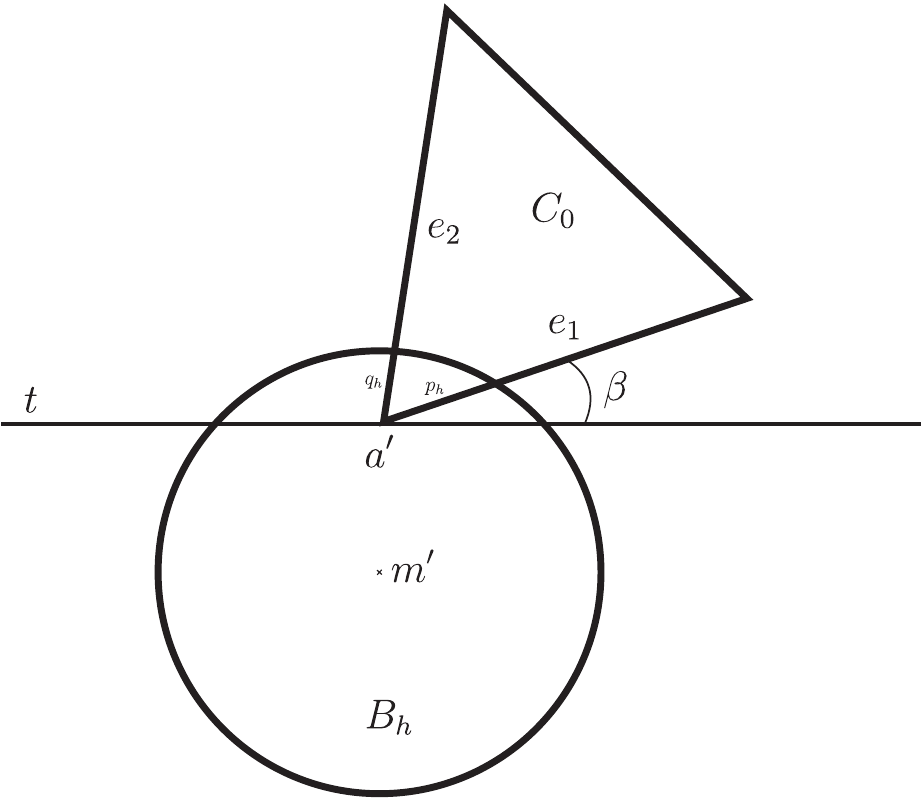}
\caption{The plane $\pi_{a_1+h}$
(formally, here $C_0 = \shadow(C_+ \cap \pi_{a_1}) \cap \pi_{a_1+1}$), \hskip3pt $d=3$.}
\label{fig4r}
\end{figure}
From $a'$ we move first along the edge $e_1$ and then possibly continue beyond it in the same direction until we hit
$\partial B_h$, say at distance $p_h$ from $a'$. Then
\begin{equation} \label{eq:p2r}
(m_1 - a_1 - h)^2 +  p_h^2 \cos^2\beta + (R + p_h \sin\beta)^2 = r^2.
\end{equation}
Subtracting \eqref{eq:p1r}, we get
\begin{equation*}
p_h^2 + 2p_h R \sin\beta - 2(m_1-a_1)h + h^2 = 0.
\end{equation*}
We rewrite this as a quadratic equation in $1\slash p_h$ which we solve, getting
$$
p_h = \frac{2(m_1-a_1)h - h^2}{\sqrt{R^2 \sin^2\beta + 2(m_1-a_1)h - h^2}+R \sin\beta}.
$$
Since $h \le r/\sqrt{3} < m_1-a_1 \le r$, we see that
\begin{equation} \label{phr}
  p_h \simeq \frac{(m_1-a_1)h}{\sqrt{(m_1-a_1)h}+ R \sin\beta} \simeq \frac{rh}{\sqrt{rh}+ R\sin\beta}
	\simeq \sqrt{rh} \wedge \frac{rh}{R\sin\beta}.
\end{equation}

We next repeat the above, but moving in the direction of $e_2$ until we leave $B_h$,
after covering a distance $q_h$, say. The same argument applies, but instead of $\beta$ we have
$\beta + \pi/3 \in [\pi/3,2\pi/3]$. The result is
\begin{equation} \label{qhr}
q_h \simeq \sqrt{rh} \wedge \frac{rh}{R}.
\end{equation}

We now estimate the measure of $\widetilde{B}$ from below.
Consider for the time being only $h \in (1/2,1)$. Then $p_h \simeq p_1$ and $q_h \simeq q_1$.
Thus we can find one point on each edge $e_1$ and $e_2$ belonging to the closure of $B_h \cap C_h$ whose distance from $a'$
is comparable to $a_1 \wedge p_1$ and $a_1 \wedge q_1$, respectively (recall that $|e_1|=|e_2|\simeq a_1$).
The triangle formed  by these two points and $a'$ is also
contained in $B_h \cap C_h$ by convexity, and its area is comparable to $(a_1 \wedge p_1)(a_1\wedge q_1)$.
Integrating over $1/2 < h < 1$, we see that
\begin{equation} \label{eq:p3r}
\nu(\widetilde{B}) \gtrsim (a_1 \wedge p_1) (a_1 \wedge q_1) e^{-a_1}.
\end{equation}

Next, we consider all $h \in (0, r+(m_1-a_1))$. We shall need the following.

\begin{propo} \label{lem:fcl}
There is an increasing family of parallelograms $\{\mathcal{P}_h : 0 < h < r+(m_1-a_1)\}$ in $\mathbb{R}^2$
with sides parallel to $e_1$ and $e_2$ and side lengths controlled (up to multiplicative absolute constants)
by $(a_1 + h) \wedge p_h$ and $(a_1 + h) \wedge q_h$, respectively, in case $0 < h \le r/3$, and by $r$ in case $h > r/3$, such that
$B_h \cap C_h \subset \mathcal{P}_h$ for  $0 < h < r+(m_1-a_1)$.
\end{propo}

\begin{proof}
Consider first $h \le r/3$. The triangle $C_h$ is a concentric scaling of $C_0$,
and all its points have a distance of at most $\sqrt{2}\,h$ from $C_0$.
In particular, $C_h$ has a vertex $a_h'$, corresponding to $a'$,
which is at the distance $\sqrt{2}\,h \sin(\beta+\pi/6)$ from $t$, and at the distance $\sqrt{2}\,h \cos(\beta + \pi/6)$ from the line
perpendicular to $t$ and passing through $a'$ (and $m'$); see Figure \ref{fig5r}.
\begin{figure}[ht]
\includegraphics[height=0.5\textwidth]{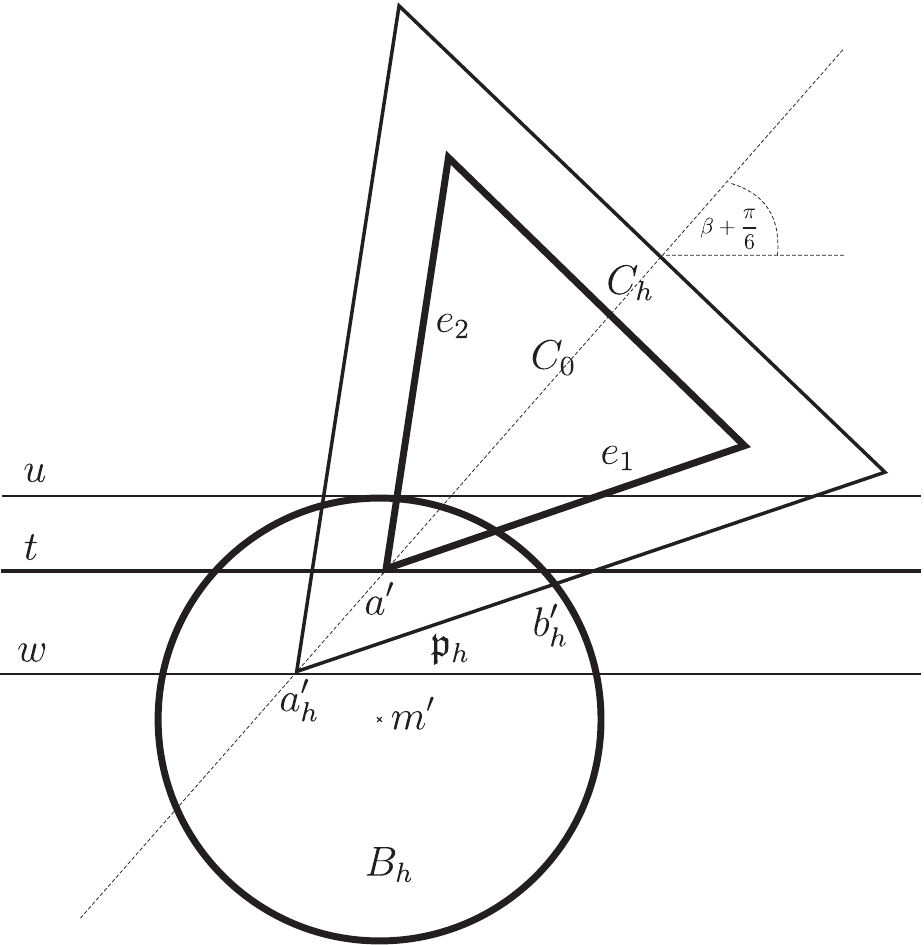}
\caption{The plane $\pi_{a_1+h}$ (formally, here $C_0 = \shadow(C_+ \cap \pi_{a_1}) \cap \pi_{a_1+1}$), \hskip3pt $d=3$.}
\label{fig5r}
\end{figure}
Bring in the ``vertical coordinate''
\begin{equation}\label{vertical}
  \tau(x') = \langle x'-m', a'-m'\rangle
\end{equation}
in the plane $\pi_{a_1+h}$.

We take as $\mathcal{P}_h$ the smallest open parallelogram having one vertex at $a'_h$, sides parallel to $e_1$ and $e_2$,
and containing $B_h \cap C_h$. Then we must show that the side lengths of $\mathcal{P}_h$ are controlled by
$(a_1+h) \wedge p_h$ and $(a_1+h) \wedge q_h$. We shall separate the cases depending when $a'_h$ lies below or above
the level of $m'$, see Figure \ref{fig5r}.

Assume first that $\tau(a'_h) \le 0$, i.e., $a'_h$ does not exceed the level of $m'$. This means that
$\sqrt{2}\,h \sin(\beta+\pi/6) \ge |a' - m'| = R$, which implies $h \ge R/\sqrt{2}$. In this situation,
see \eqref{phr} and \eqref{qhr}, one has $p_h \simeq \sqrt{rh} \simeq q_h$. On the other hand,
the radius of $B_h$ is also comparable with $R+ q_h \simeq \sqrt{rh}$.

Now, observe that the edges of $\mathcal{P}_h$ have lengths controlled by the quantity
$$
(\textrm{side length of}\; C_h) \wedge (\textrm{radius of} \; B_h ),
$$
thus by $(a_1+h)\wedge \sqrt{rh} \simeq (a_1+h)\wedge p_h \simeq (a_1+h)\wedge q_h$, as desired.

Next, assume that $\tau(a'_h) > 0$, i.e., $a'_h$ is above the level of $m'$.
We shall construct a parallelogram $\mathcal{P}_h^*$ containing $\mathcal{P}_h$, having vertex at $a'_h$ and sides
parallel to $e_1$ and $e_2$, whose side lengths satisfy the desired estimates. Clearly, this will be enough for our purpose.

In the plane $\pi_{a_1+h}$, let $w$ be the line through $a'_h$ parallel to $t$.
Define $u$ to be the line parallel to $w$ given by $u = \{x' : \tau(x') = (\sup_{B_h}\tau) \wedge (\sup_{C_h}\tau)\}$.
Observe that two cases may occur (call them (a) and (b), respectively):
$u$ is tangent to $B_h$ (if the last minimum is realized by $\sup_{B_h}\tau$, see Figure \ref{fig5r})
or $u$ passes through the vertex of $C_h$ of maximal distance from $w$.

The intersection $B_h \cap C_h$ is contained in the band between $w$ and $u$.
In case (a), the width of this band, see Figure \ref{fig5r}, is not larger than (actually comparable with)
$q_h + \sqrt{2}\,h \sin(\beta+\pi/6)$, and this quantity, in view of \eqref{qhr}, is comparable to $q_h$.
In case (b), the width of this band is comparable with the side length of $C_h$, i.e., with $a_1 + h$.

Now consider the segment along the $e_2$ direction with one endpoint at $a_h'$, and whose
other endpoint lies on $u$. Since $e_2$ forms the angle $\beta+\pi/3$ with $u$, which is separated from $0$, the segment
in question has length comparable with the width of the band, thus with $(a_1+h) \wedge q_h$.
We take this segment as a side of our $\mathcal{P}_h^*$.

As the other side of $\mathcal{P}_h^*$ we shall take the segment along the $e_1$ direction with one endpoint
at $a_h'$, and the other endpoint $b_h'$ lies either on the boundary of $B_h$, inside the band, or is the vertex of
$C_h$ in case $B_h$ is so large that $\partial B_h$ does not cross this ($e_1$-directed) side of $C_h$.
See again Figure \ref{fig5r}.
Denote by $\mathfrak{p}_h$ the length of this segment.
Clearly, $\mathfrak{p}_h$ is comparable with $a_1+h$ in case $b_h'$ is the vertex.
Assuming the other case $b_h' \in \partial B_h$, we will show that $\mathfrak{p}_h$ is comparable to $p_h$,
a fact that is intuitively clear from the picture.
Since $\mathcal{P}_h^*$ just defined contains$^\ddag$ $B_h \cap C_h$,
this will finish the reasoning when $h\le r/3$.
\footnotetext{$\ddag$
This inclusion is seen from the geometry of the situation, see Figure \ref{fig5r}.
Perhaps the least obvious point is to ensure that in the case when $b_h' \in \partial B_h$
the edge of $\mathcal{P}_h^*$ starting at $b_h'$ and parallel to $e_2$ does not cross $B_h$.
Indeed, this is true since the outward normal of $B_h$ at $b'_h$ enters into $C_h$. Thus the angle between this normal
and $e_2$ is less than $\pi/3$, and the inclusion follows.}

Observe that, cf.\ \eqref{eq:p2r},
\begin{equation*} 
(m_1-a_1-h)^2 + \big( \mathfrak{p}_h \cos\beta - \sqrt{2}\,h\cos(\beta+\pi/6) \big)^2
	+ \big( \mathfrak{p}_h \sin\beta + R - \sqrt{2}\,h\sin(\beta+\pi/6)\big)^2 = r^2.
\end{equation*}

Subtracting \eqref{eq:p1r} and solving for $1\slash \mathfrak{p}_h$ (see the analysis leading from \eqref{eq:p2r} to \eqref{phr}),
we get after some elementary computations and applications of basic trigonometric identities
$$
\mathfrak{p}_h = \frac{2(m_1-a_1)h + 2\sqrt{2}R h \sin(\beta+\pi/6) - 3h^2}{\sqrt{R^2\sin^2\beta + 2h(m_1-a_1)
	+ \sqrt{2} R h \cos\beta - 3h^2/2} + R\sin\beta - \sqrt{2}\sqrt{3}\,h/2}.
$$
Then, recalling that $r/\sqrt{3} < m_1-a_1 \le r$, $R < r$ and $h \le r/3$, we arrive at
$$
\mathfrak{p}_h \simeq \frac{rh}{R\sin\beta} \wedge \sqrt{rh} \simeq p_h.
$$

Considering $h > r /3$, take as $\mathcal{P}_h$ the smallest (open) parallelogram,
with sides parallel to $e_1$ and $e_2$, containing both $\mathcal{P}_{r/3}$ and $B_{m_1-a_1}$.
This parallelogram has side lengths comparable to $r$, by the geometry of the situation.

The fact that the family $\{\mathcal{P}_h : h > 0\}$ is increasing is clear from the construction.
Proposition \ref{lem:fcl} follows.
\end{proof}

In view of \eqref{eq:p3r}, for $\mathcal{P}_h$ from Proposition \ref{lem:fcl} we have the bound
\begin{equation*}
\frac{|\mathcal{P}_h|}{e^{a_1} \nu(\widetilde{B})} \lesssim
	\begin{cases}
		\frac{[(a_1+h)\wedge p_h][(a_1+h)\wedge q_h]}{(a_1 \wedge p_1)(a_1\wedge q_1)}, & 0 < h \le r\slash 3, \\
		\frac{r^2}{(a_1 \wedge p_1)(a_1\wedge q_1)}, & h > r/3.
	\end{cases}
\end{equation*}
To estimate the right-hand side here we use \eqref{phr} and \eqref{qhr}, and apply an elementary analysis of cases.
Considering $h \le r/3$, if $a_1 \wedge p_1 = a_1$, then (recall that $a_1 > 2$)
$$
\frac{(a_1+h)\wedge p_h}{a_1 \wedge p_1} \le \frac{a_1 + h}{a_1} < 1 + h;
$$
if $a_1 \wedge p_1 = p_1$, then
$$
\frac{(a_1+h)\wedge p_h}{a_1 \wedge p_1} \le \frac{p_h}{p_1} \simeq
	\frac{\frac{rh}{R\sin\beta} \wedge \sqrt{rh}}{\frac{r}{R\sin\beta}\wedge \sqrt{r}}
	\le h \vee \sqrt{h} \lesssim 1 + h.
$$
For $h > r/3$ we have
$$
\frac{r}{a_1 \wedge p_1} \simeq \frac{r}{a_1 \wedge \frac{r}{R\sin\beta} \wedge \sqrt{r}}
	\le r \vee (R\sin\beta) \vee \sqrt{r} \lesssim 1 + h.
$$
The factors involving $q_h$ are treated similarly. Thus we arrive at the key bound
\begin{equation} \label{eq:p7r}
\frac{|\mathcal{P}_h|}{e^{a_1} \nu(\widetilde{B})} \lesssim 1 + h^2, \qquad h > 0,
\end{equation}
uniformly in $a_1$ and $\widetilde{B}$.

We are now in a position to apply the slicing argument.
Let $M$ be the part of the maximal operator \eqref{max_B} under consideration.
As in dimension $2$, we define $S_{i}$ for $i \ge 1$ as $\{x \in C_+ : i < x_1 \le i+1\}$, and in $S_i$,
$e^{-i-1}dx \le d\nu(x) < e^{-i}dx$.
It is enough to prove that for some constant $\delta > 0$
\begin{equation} \label{4.X}
\int_{S_j} \big[M(f\chi_{S_i})\big]^p\, d\nu \lesssim e^{-\delta |j-i|} \int_{S_i} f^p\, d\nu, \qquad i,j \ge 1,
\end{equation}
see \eqref{nkey_est} and the preceding comments.

To prove \eqref{4.X},
let $i,j \ge 1$. Let $x \in \widetilde{B} \cap S_j$ and $y \in \widetilde{B} \cap S_i$.
Proposition \ref{lem:fcl} tells us that $x'$ and $y'$ are contained in a certain parallelogram $\mathcal{P}_h$, and
both parallelograms are contained in the one given by Proposition \ref{lem:fcl} with $h = (i-a_1 + 1) \vee (j-a_1+1)$;
notice that here $i,j \ge a_1 - 1$. Define
$$
\mathcal{M}'g(z) = \sup \frac{1}{|\mathcal{P}|} \int_{\mathcal{P}} |g(w)|\, dw,
$$
for any locally integrable function $g$ in $\mathbb{R}^2$, where the supremum is taken over
all parallelograms $\mathcal{P}$ containing $z$ and with sides parallel to two sides of the triangle $C_0$, and
$|\mathcal{P}|$ denotes the area of $\mathcal{P}$.
Then we can write the estimate
\begin{equation} \label{eq:p8r}
\frac{1}{\nu(\widetilde{B})} \int_{\widetilde{B}} \chi_{S_i}(y) f(y)\, d\nu(y)
 \le \frac{1}{\nu(\widetilde{B})}\, e^{-i}\, |\mathcal{P}_{i \vee j - a_1 + 1}| \int_i^{i+1} \mathcal{M}'f(y_1,x')\, dy_1.
\end{equation}
Note that $\mathcal{M}'$ is bounded on $L^p(\mathbb{R}^2)$ for
$1 < p < \infty$. Indeed, $\mathcal{M}'$ splits naturally into three components,
each determined by two edges of $C_0$. Then a linear transformation makes each component coincide with
the strong maximal operator $M_{\mathrm{str}}$ in $\mathbb{R}^2$.

Combining \eqref{eq:p8r} with \eqref{eq:p7r} we obtain
$$
\frac{1}{\nu(\widetilde{B})} \int_{\widetilde{B}} \chi_{S_i}(y) f(y)\, d\nu(y)
 \lesssim
e^{a_1-i} \big[ 1+ (i \vee j - a_1 + 1)^2 \big] \int_i^{i+1} \mathcal{M}'f(y_1,x')\, dy_1.
$$
This is an analogue of \eqref{n420}. From here one proceeds as before, arguing as done after \eqref{n420},
getting $L^p(d\nu)$-boundedness of the considered part of our maximal operator.

\medskip

\noindent \textbf{Case II:} $a(\widetilde{B})$ lies on a face of $C_+$.

Then $a'$ is an inner point of a side of the triangle $C_0$; see Figure \ref{fig6r}.
\begin{figure}[ht]
\includegraphics[height=0.4\textwidth]{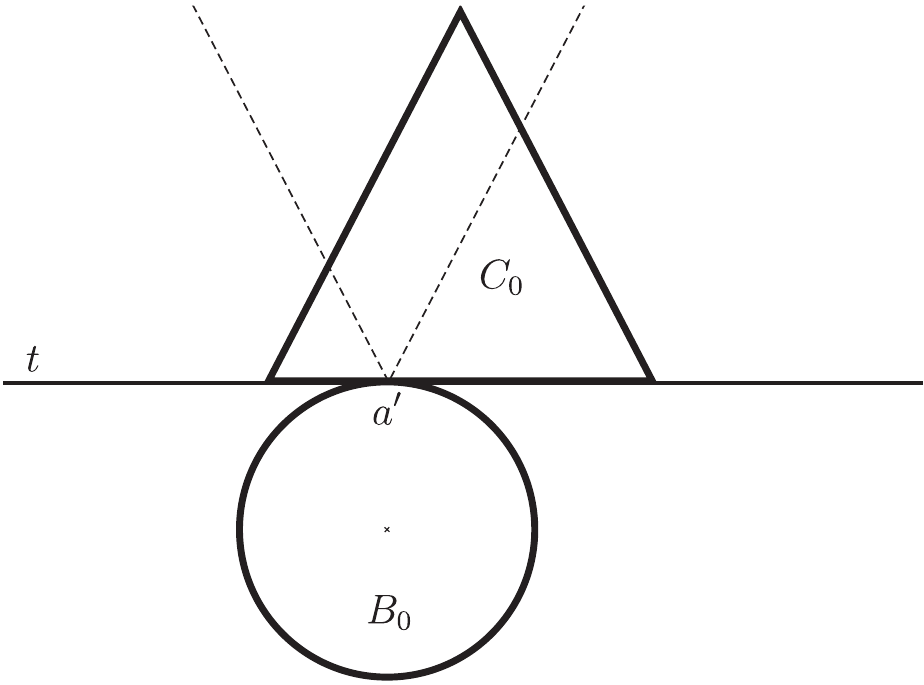}
\caption{The plane $\pi_{a_1}$ in Case II, $d=3$.}
\label{fig6r}
\end{figure}
We split $C_0$ into its intersections with three two-dimensional cones, by introducing two rays from $a'$ forming
angles of $\pi/3$ with the side of $C_0$. Then we apply the arguments from Case I,
using instead of $C_0$  each of these three intersections,
with $\beta =0$ twice and with $\beta = \pi/3$ once, as seen in Figure \ref{fig6r}.
That intersection which has  $\beta = \pi/3$ is not a triangle but a parallelogram.
But notice that the $h$-expansion of this parallelogram, analogous to $C_h$ in Case I, will necessarily
be contained in the analog of the parallelogram $\mathcal{P}_h$ constructed in Proposition~\ref{lem:fcl}.
To get the lower estimate \eqref{eq:p3r}, it is enough to argue as in Case I for the larger of the two
intersections with $\beta =0$. In each of the three intersections, we can now follow the pattern of Case I for all
upper estimates of integrals, and divide by $\nu(\widetilde{B})$.

This ends the case of dimension $3$.

\bigskip

\textbf{Dimension $\boldsymbol{d=4}$.}
We largely follow the three-dimensional argument. Recall that
$$
\widetilde{B} = \widetilde{B}(m,r) := \mathbf{B}(m,r) \cap C_+.
$$
The assumptions \eqref{obs12} remain in force. As in dimension three, we define for 
$0 \le h < r + (m_1 - a_1)$
\begin{equation}\label{ch}
C_h = \big\{ x' \in \mathbb{R}^3 : (a_1+h,x') \in C_+\big\}
\end{equation}
which is an open regular tetrahedron of side  $(a_1+h)\,\sqrt{8}$, and 
$$
B_h = \big\{ x' \in \mathbb{R}^3 : (a_1+h,x') \in \mathbf{B}(m,r) \big\}.
$$
Observe that \eqref{obs12} implies $a' \in \partial C_0$.
The radius of the ball $B_h$ will be denoted by $R_h$, and as before we write $R$ for $R_0$.

In $\pi_{a_1}$, which we identify with $\mathbb{R}^3$,
we now let $T$ be the tangent plane of the ball $B_0$ passing through  $a'$.
Moreover, $T_+$ will denote that closed half-space in $\pi_{a_1}$ whose boundary is $T$ and which contains $C_0$.

Instead of \eqref{m1-a1}, we now have
\begin{equation}\label{m1-a14}
	1 < r/2 < m_1-a_1 \le r.
\end{equation}
The equality \eqref{eq:p1r} remains valid and implies
\begin{equation} \label{m}
	R^2 = r^2 - (m_1-a_1)^2 =  (r+(m_1-a_1))(r-(m_1-a_1)) \simeq r(r-(m_1-a_1)).
\end{equation}
When $h<r$, we similarly get for $R_h$ in view of \eqref{m1-a14}
\begin{equation} \label{mm}
	R_h^2 = r^2 - (m_1-a_1-h)^2 =  (r+(m_1-a_1-h))(r-(m_1-a_1)+h) \simeq r(r-(m_1-a_1)+h).
\end{equation}
We also have
\begin{equation} \label{mmm}
	R_h^2 - R^2 = (m_1-a_1)^2 - (m_1-a_1-h)^2 = 2(m_1-a_1)h -h^2 \lesssim rh,
\end{equation}
the last step by \eqref{m1-a14}.

\medskip

The ``vertical coordinate'' $\tau$ in $\mathbb{R}^3$ is defined by \eqref{vertical}, as in the  three-dimensional case.

\medskip

We will need some angles connected with a regular tetrahedron. The angle at a vertex between an edge and the axis of symmetry
from that vertex is $\gamma$, where $\sin \gamma = 1/\sqrt 3$, and the angle between two faces of the tetrahedron is $2\gamma$.
Further, the angle between a face and an edge not in that face is $\kappa$, where  $\sin \kappa = \sqrt {2/3}$. Using this last angle, 
one finds that the ratio between the height and the edge of the tetrahedron is $ \sqrt {2/3} > 1/2$;
the height is the distance between a vertex and the opposite face.

\medskip

\noindent \textbf{Case I:}  $a'$ is a vertex of  $C_0$.

In $\mathbb{R}^3$, the point $a'$ is now an endpoint of three edges $e_1$, $e_2$ and $e_3$
of the tetrahedron $C_0$. Let $\beta_i,\; i= 1,2,3$, denote the angle at $a'$ between $e_i$ and the plane $T$.
Then $0 \le \beta_i \le \pi/2$, and at most two of the $\beta_i$ can be small.

Clearly $a'$ is an inner point of the ball $B_h$ when $ 0 < h < r + (m_1 - a_1)$. We consider for $i= 1,2,3$
the intersection of  $B_h$ and the ray in the direction of $e_i$ emanating from  $a'$.
Let $p_h^{i}$ be the length of this intersection. We can determine the $p_h^{i}$
exactly like $p_h$ in dimension three, and instead of \eqref{phr} we get for $0 < h< r/2$
\begin{equation} \label{phi}
  p_h^{i} \simeq \sqrt{rh} \wedge \frac{rh}{R\sin\beta_i}, \qquad i= 1,2,3.
\end{equation}

The argument leading to \eqref{eq:p3r} also carries over, so that
\begin{equation} \label{lower}
\nu(\widetilde{B}) \gtrsim \big(a_1 \wedge p_1^{1}\big) \,  \big(a_1 \wedge p_1^{2}\big) \, \big(a_1 \wedge p_1^{3}\big) \, e^{-a_1}.
\end{equation}

As before, $a'_h$ denotes the vertex of $C_h$ that corresponds to $a'$; one finds that the distance from $a'$
to $a'_h$ is $\sqrt 3\,h$.

Let $\mathcal{P}_h \subset \mathbb{R}^3$ for $0 < h  \le r/2$ be the minimal parallelepiped
containing $B_h \cap C_h$ which has one vertex at  $a'_h$ and  edges parallel to $e_1$, \hskip1pt $e_2$ and $e_3$.
Then $\mathcal{P}_h$ increases with $h$.

\begin{propo} \label{parallelep}
For $0 < h  \le r/2$,
  the sides of $\mathcal{P}_h$ are bounded by constant times  $(a_1+h) \wedge p_h^{i}$,
  \hskip3pt $i= 1,2,3$.
\end{propo}

To prove this, we fix $h \in (0, r/2]$ and deal first with the simple case when $h \ge c_0 R$, for some small constant
$c_0 > 0$ to be determined. Then the $p_h^{i}$  are all of magnitude $\sqrt{rh}$, and \eqref{mmm} implies
$$
  R_h^2  \lesssim R^2 + rh \lesssim rh,
$$
the last step since here $R \lesssim h$. Thus $R_h \lesssim \sqrt{rh}$.

Comparing the sides of  $\mathcal{P}_h$ with the minimum of $R_h$ and the side of $C_h$, we arrive at the conclusion
of the proposition, when $h \ge c_0 \,R$.

\smallskip

Consider now the remaining case $0 < h < c_0 R$, and observe that then $a'_h \in B_h.$
Let  $i\in \{ 1,2,3\}$. We define $\rho_i$ as the ray  parallel with $e_i$, with endpoint at $a'_h$
and contained in the half-space $ \{x': \tau(x') \ge \tau(a'_h) \} $.
If  $\sin\beta_i \ge 1/32$, we denote by  $b'_i$ the point of intersection of $\rho_i$ and $\partial H$, where $H$ is the half-space
$$
  H = \Big\{x': \tau(x') \le \sup_{B_h}\tau\Big\}.
$$
When $\sin\beta_i < 1/32$, we define $b'_i$ similarly, but now with the intersection point of $\rho_i$ and $\partial {B_h}$.
Finally, let $v_i$ be the vector  $b'_i - a'_h$, which is parallel with  $e_i$. See Figure \ref{fig9999}.

Define now
\begin{equation*}
  \mathcal{P}'_h = \left\{a'_h + \sum_1^3 \alpha_iv_i:\; 0 \le \alpha_i \le 1,\; i = 1,2,3 \right\},
\end{equation*}
a parallelepiped with one vertex at $a'_h$ and side lengths $|v_i|$. It is increasing in $h$.

We will need the following two lemmas, whose proofs are given after the end of the proof of Proposition \ref{parallelep}.

\begin{lemma} \label{pp}
If $0 < h< c_0R$ with $c_0$ small enough, then for $i= 1,2,3$
   \begin{equation*}
    |v_i|   \lesssim  {p}_h^{i},
   \end{equation*}
and if moreover $\sin\beta_i < 1/32$, then
   \begin{equation*}
    |v_i| \ge 4h.
   \end{equation*}
\end{lemma}

\begin{lemma}\label{inclusion}
If $0 < h<c_0R$ with $c_0$ small enough, then
\begin{equation*}
	B_h \cap C_h \subset \mathcal{P}'_h.
\end{equation*}
\end{lemma}

Given these lemmas, and still assuming $0 < h < c_0 R$, let $ \mathcal{P}''_h$ be the minimal parallelepiped with one vertex at
$a'_h$ that contains $C_h$. Then the parallelepiped $\mathcal{P}^*_h = \mathcal{P}'_h \cap \mathcal{P}''_h$ will contain
$B_h \cap C_h$ because of Lemma~\ref{inclusion}. From Lemma~\ref{pp} and the fact that the sides of
$ \mathcal{P}''_h$ are of order of magnitude $a_1+h$, it follows that the sides of $\mathcal{P}^*_h$
are as stated in Proposition \ref{parallelep}. The minimality of $\mathcal{P}_h$ shows that
$\mathcal{P}_h \subset \mathcal{P}^*_h$, and this concludes the proof of Proposition \ref{parallelep}.

\medskip

In the proofs of the two lemmas, we will denote by $\omega$ the angle at $a'$ between the central axis of $C_0$ emanating from
$a'$ and the plane $T$.
Notice that $\omega \gtrsim 1$, since $\omega$ is at least as large as the angle between the central axis and a face of $C_0$.

\begin{figure}[ht]
\includegraphics[height=0.5\textwidth]{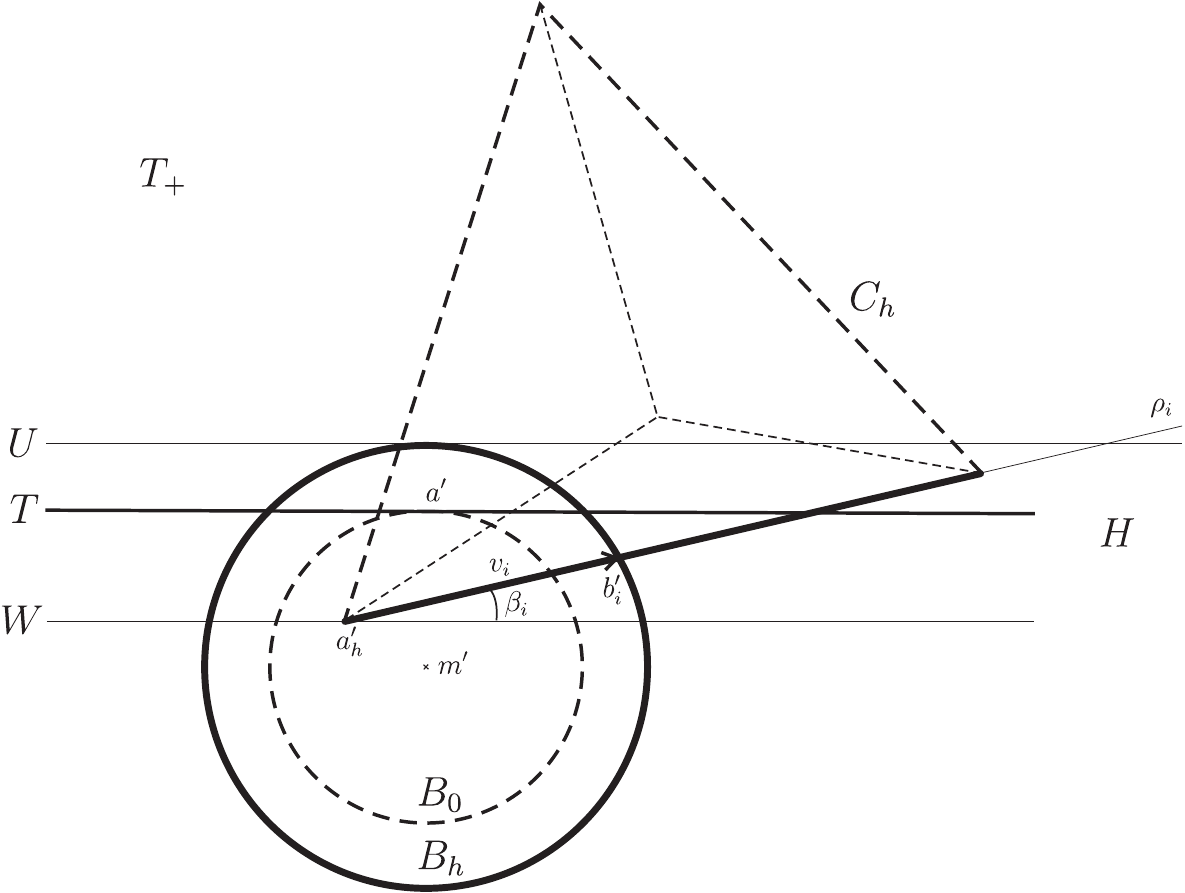}
\caption{The plane perpendicular to $T$ containing $a'_h$ and $b'_i$, \hskip2pt $d=4$.}
\label{fig9999}
\end{figure}

\begin{proof}[Proof of Lemma \ref{pp}]
Consider first the case  $\sin\beta_i \ge 1/32$.
The vertical distance $\tau(b'_i) - \tau(a'_h)$ is $R_h - R + \sqrt 3\, h \sin \omega$, see Figure \ref{fig9999}.
This gives an expression for $|v_i|$, and then we use in turn \eqref{mmm}, \eqref{mm}, \eqref{m} and then \eqref{phi}. As a result,
 \begin{multline*}
   |v_i| = \frac{R_h - R + \sqrt 3\, h \sin \omega}{\sin\beta_i} \lesssim \frac{R_h^2 - R^2}{R_h} + h \lesssim
   \frac{rh}{\sqrt{r(r-(m_1-a_1)+h)}} \\
   \lesssim \frac{rh}{\sqrt{r(r-(m_1-a_1))}} \wedge \sqrt{rh}
   \simeq \frac{rh}{R} \wedge \sqrt{rh} \simeq p_h^i.
 \end{multline*}

In the opposite case  $\sin\beta_i < 1/32$, the quantity  $|v_i|$ is the length of a segment from $a'_h$ to a point
on $\partial B_h$. The segment forms an angle $\beta_i$ with the plane $W = \{x': \tau(x') = \tau(a'_h)\}$
(and is on the same side of $W$ as the point $a'$), as seen in Figure \ref{fig9999}.

Project this segment and also the central axis of $C_h$ starting at $a'_h$ orthogonally onto the plane $W$.
Let $\theta$ denote the angle between these two projections at their common point $a'_h$.

Since the endpoint of the segment is on $\partial B_h$, the following equation will have the positive solution $z = |v_i|$,
and also a negative solution. We temporarily write $\ell = R - \sqrt 3 \,h\sin \omega$,
which is the vertical distance between $m'$ and $a'_h$. The equation is
\begin{equation*}
   (\ell + z\sin \beta_i)^2 + (-\sqrt 3 \,h\cos \omega + z \cos\beta_i \cos\theta)^2 + (z \cos\beta_i \sin\theta)^2 = R_h^2,
\end{equation*}
or simplified
\begin{equation*}
  z^2 + 2Kz  + L = 0,
\end{equation*}
where $K = \ell \sin \beta_i - \sqrt 3 \,h \cos \omega \cos\beta_i \cos\theta$ and
$L =\ell^2 + 3h^2\cos^2 \omega - R_h^2$. We consider this equation for all $\theta \in [0, \pi]$.
Since the two roots of the equation have opposite signs, the constant term $L$ is negative, which can also be seen geometrically.
Let us now vary only $\theta$, and write the positive solution as $z = z(\theta)$.
Differentiating the equation with respect to $\theta$, we get
\begin{equation*}
    (z+K)\,\frac{dz}{d\theta} = - z \sqrt 3\, h \cos \omega \cos\beta_i \sin\theta.
\end{equation*}
Since $z$ is the positive solution of the equation, $z+K$ equals the square root that appears in the well-known formula for the
solutions, so it is positive. Thus $dz/d\theta < 0$ for $0 < \theta < \pi$. It follows that  the minimal and maximal values of
$z(\theta)$ are $z(\pi)$ and $z(0)$, respectively, so that $z(\pi) \le|v_i| \le z(0)$.
We now rewrite the equation with these two values of $\theta$, and replace $K,\:L$ and also $\ell$ by their explicit expressions. 
Using  some elementary trigonometry, one obtains the result
\begin{equation*}
  z^2 + 2\big(R \sin \beta_i  \pm  \sqrt 3 \,h \cos(\omega \pm \beta_i)\big)z
  - \big(R_h^2 - R^2 + 2 \sqrt 3 \,Rh \sin \omega -3h^2\big) = 0,
\end{equation*}
where  the $\pm$ signs should be read as plus for $z(\pi)$ and minus for $z(0)$.
We now solve this equation for $1/z$, denoting
$$
   K_* = R \sin \beta_i  \pm  \sqrt 3 \,h \cos(\omega \pm \beta_i) \qquad \mathrm{and} \qquad
   L_* = R_h^2 - R^2 + 2 \sqrt 3 \,Rh \sin \omega -3h^2.
$$
The positive solution $z$ is given by
\begin{equation} \label{fraction}
   z = \frac{L_* }{K_* + \sqrt{K_*^2 +L_*}}.
\end{equation}

We estimate  the numerator and the denominator in \eqref{fraction} from above and below,
choosing $c_0$ small enough whenever needed. Because of \eqref{mmm}, we find
\begin{equation}\label{L+}
   L_* \le 2rh + 4Rh  \le  6rh  \lesssim rh
\end{equation}
and
\begin{equation}\label{L-}
   L_* \ge  2(m_1 - a_1)h -4h^2 \ge rh - 4c_0rh \ge rh/2,
\end{equation}
where we also used \eqref{m1-a14}. Further, \eqref{L+} implies that
\begin{equation}\label{d+}
  K_* + \sqrt{K_*^2 +L_*} \le  2|K_*| + \sqrt{L_*}  \le 2R \sin \beta_i + 2\sqrt 3 \,h +  \sqrt {6rh} 
	\le 2R \sin \beta_i + 3 \sqrt {rh}.
\end{equation}
From \eqref{L-}, we obtain
\begin{equation}\label{d-}
  K_* + \sqrt{K_*^2 +L_*} \ge  R \sin \beta_i - \sqrt 3 \,h +  \sqrt {rh/2} \gtrsim R \sin \beta_i +  \sqrt {rh}.
\end{equation}
These four inequalities hold whether the $\pm$ signs are read as plus or minus.

Combining \eqref{L+} and \eqref{d-} with \eqref{fraction}, we conclude that
\begin{equation*}
 |v_i|  \lesssim \frac{rh}{R\sin \beta_i} \wedge \sqrt{rh} \simeq p_h^i,
\end{equation*}
because of \eqref{phi}. If  $\sin\beta_i < 1/32$, \eqref{L-} and \eqref{d+} similarly yield
\begin{equation*}
 |v_i|  \ge \frac{rh/2}{2R \sin \beta_i + 3 \sqrt {rh}} \ge
 \frac12 \, \frac{rh/2}{(2R \sin \beta_i) \vee \big(3 \sqrt {rh}\big)} \ge
  \frac{rh}{8R \sin \beta_i} \wedge \frac{\sqrt {rh}} {12} \ge 4h.
\end{equation*}
The last two formulas end the proof of Lemma \ref{pp}.
\end{proof}

\begin{proof}[Proof of Lemma \ref{inclusion}]
Any point $x \in C_h$ can be written  $x = a'_h + \sum_1^3 \alpha_j v_j$ with $\alpha_j\ge 0$.
Assume now that  $x \in B_h \cap C_h$. We will show that  $\alpha_j \le 1$ for each $j$, so that
$x \in \mathcal{P}'_h$.  Thus we fix $j\in \{ 1,2,3\}$. Observe that to prove the inequality
$\alpha_j \le 1$, we may assume that  $\alpha_j \ge 1$, since the opposite case is clear.

Since $x= a'_h + \sum_1^3 \alpha_i v_i \in H$ and the function $\alpha_i \mapsto \tau(a'_h + \sum_1^3 \alpha_i v_i)$
is nondecreasing for each $i$, the point $ a'_h +\alpha_j v_j$ is also in $H$.
If $\sin\beta_j \ge 1/32$, this implies $\alpha_j \le 1$, by the definitions of $v_j$ and $b'_h$.

When instead $\sin\beta_j < 1/32$, we will similarly show that $\alpha_j \le 1$ by proving that  $ a'_h +\alpha_j v_j \in B_h$.
We know that $a'_h + \sum_1^3 \alpha_j v_j\in B_h$,
so it is enough to verify that the distance $|a'_h  + \sum_1^3 \alpha_i v_i-m'|$ is increasing in
$\alpha_i$ for $i \ne j$. But
\begin{multline*}
  \Big|a'_h + \sum_1^3 \alpha_i v_i-m' \Big|^2  \\
  =  \big|a'_h -m'\big|^2 + 2\sum_1^3 \alpha_i \langle a'_h -m', v_i \rangle + \sum_1^3 \alpha_i^2 |v_i|^2
		+ 2\sum_{1 \le i < k  \le 3} \alpha_i \alpha_k \langle  v_i, v_k \rangle,
\end{multline*}
and here all the terms to the right except possibly the second one are nondecreasing in $\alpha_i$. Further,
$\langle  v_i, v_k \rangle = |v_i| |v_k|/2$. Consider for $i \ne j$ the following  two terms from the right-hand side
\begin{equation}\label{terms}
  2 \alpha_i \langle a'_h -m', v_i \rangle +  \alpha_i \alpha_j |v_i| |v_j|.
\end{equation}
Now
\begin{equation*}
 \langle a'_h -m', v_i \rangle = -\sqrt 3\, h\cos \omega \,|v_i| \cos \beta_i  +  (R-\sqrt 3\, h\sin \omega)|v_i| \sin \beta_i,
\end{equation*}
so \eqref{terms} equals
\begin{equation*}
  \alpha_i |v_i|
  \left(-2\sqrt 3\, h\cos \omega \cos \beta_i +  2\big(R-\sqrt 3\, h\sin \omega\big)\sin \beta_i + \alpha_j  |v_j|\right).
\end{equation*}
It is  enough to verify that the three terms in this parenthesis have a positive sum.
The middle term is positive, since $c_0$ is small.
Recall that we assumed  $\alpha_j  \ge 1$ and also $\sin\beta_j < 1/32$ which implies $ |v_j| \ge 4h$ because of
Lemma \ref{pp}. Thus the first term in the above parenthesis is dominated by the third term,
the parenthesis is positive and the expression in \eqref{terms} is increasing in $\alpha_i$, as desired.
Lemma \ref{inclusion} is proved.
\end{proof}

We can now continue Case I as in three dimensions, but using the three quantities ${p}_h^{i}$ instead of
$p_h$ and $q_h$. In the estimate \eqref{eq:p7r} the exponent of $h$ will be 3 instead of $2$.
We extend the definition of $\mathcal P_h$ by setting it equal to  the smallest parallelepiped containing
$\mathcal P_{r/2} \cap B_{m_1-a_1}$ for $r/2 < h < r+(m_1-a_1)$; cf.\ the end of the proof of
Proposition \ref{lem:fcl} in the three-dimensional case. We leave the details finishing Case I to the reader.

\medskip

\noindent \textbf{Case II:} $a'$ is an inner point of a face of $C_0$.

This face of $C_0$ is  contained in the plane $T$, and we consider the three translates
$C_0^{j},\: j= 1,2,3$, of $C_0$ along  $T$ which have a vertex at $a'$
(see Figure \ref{fig8}, where for clarity only that face of $C_0$ contained in $T$ is marked).
\begin{figure}[ht]
\includegraphics[height=0.4\textwidth]{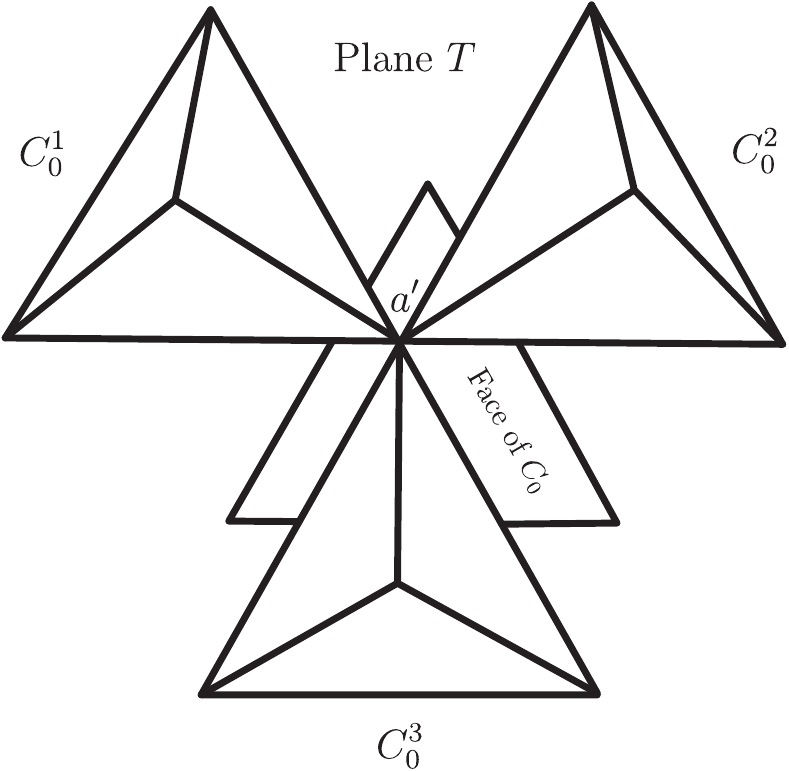}
\caption{Translates $C_0^{j}$ of $C_0$, $j= 1,2,3$, in Case II, \hskip2pt $d=4$.}
\label{fig8}
\end{figure}
The angles at $a'$ between $T$ and the edges of each $C_0^{j}$ are now $0,\;0,\;\kappa$.
The dilations  $C_h$ are given by  \eqref{ch}, and we can define for $0 \le h  < r + (m_1 - a_1)$ analogous dilations
$C_h^{j},\;j = 1,2, 3$, of the  $C_0^{j}$ by replacing in  \eqref{ch} \hskip2pt $C_+$ by the four-dimensional cone generated by
$C_0^{j} \times \{a_1\}$ and the origin. In analogy with the beginning of Case I, we consider
for each $j$ the intersection with $B_h$ of the three rays emanating from $a'$ and containing an edge of $C_h^{j}$.
As in Case I, we write  $p_h^{i},\: i= 1,2,3$, for the lengths of these intersections. The  $p_h^{i}$
will not depend on $j$, and  from \eqref{phi} we see that their orders of magnitude are
\begin{equation}\label{X}
  \sqrt{rh}, \qquad  \sqrt{rh} \qquad \mathrm{and} \qquad  \sqrt{rh} \wedge \frac{rh}{R}.
\end{equation}

At least one of the intersections $C_0\cap C_0^{j},\: j= 1,2,3$, is comparable in volume to $C_0$.
To estimate the measure of $\widetilde B$ from below, we can thus for one value of $j$ argue as in Case I with $C_0^{j}$
and $C_h^{j}$. Hence we still have the lower estimate \eqref{lower}. The corresponding upper estimate will now be verified.

In addition to $C_0^{j},\: j= 1,2,3$, we will consider a finite number of tetrahedra $C_0^{j},\,j=4,\dots, N$,
of the same size. They will all have a vertex at $a'$ and be contained in  $T_+$. We select them so that the
$C_0^{j},\,j=1,\dots N$, together cover a neighborhood of $a'$ in  $T_+$. Here $N$ will be an absolute constant.
Of the three angles at $a'$ between the plane $T$ and an edge of any $C_0^{j},\;j = 4,\dots, N$, at least one must stay away from 0,
since $C_0^{j} \subset T_+$. (In fact, the largest of these three angles is at least $\pi/4$.)  
This implies that the corresponding lengths $p_h^{i,j}$ (which will now depend also on $j$)
have orders of magnitude no larger than those in \eqref{X}.

By $C_0^{j,2}$ we denote the result of a scaling of $C_0^{j}$ centered at $a'$
by a factor of 2. Thus $a'$ is a vertex also of $C_0^{j,2}$.
The $C_0^{j,2},\;j = 1,\dots, N$, will together contain the intersection of $T_+$ and the ball
of center $a'$ and radius equal to the height of $C_0^{j,2}$. Since this height is larger than
the diameter, i.e., the edge, of $C_0^j$, we conclude that
\begin{equation*}
   \bigcup_{j=1}^N \, C_0^{j,2} \supset C_0.
\end{equation*}
            
The arguments from Case I will apply to each scaled tetrahedron $C_h^{j,2}$. In particular,
we choose as there minimal parallelepipeds $\mathcal P_h^j$ containing
$B_h \cap C_h^{j,2},\;\;j = 1,\dots, N$, which together cover $B_h \cap C_h$. The proofs of
Lemmas \ref{pp} and \ref{inclusion} and then also that of Proposition~\ref{parallelep} will go through for each $\mathcal P_h^j$,
and this allows us to conclude Case II like Case~I.

\medskip

\noindent \textbf{Case III:} $a'$ is an inner point of an edge of $C_0$.

This edge of $C_0$ will be called $e_0$. It is contained in $T$, and it is the intersection of two faces of $C_0$.
We denote by $\Pi'$ and $\Pi''$ the planes containing these faces. The angle between $\Pi'$ and $\Pi''$ is $2\gamma$.

\begin{figure}[ht]
\includegraphics[height=0.4\textwidth]{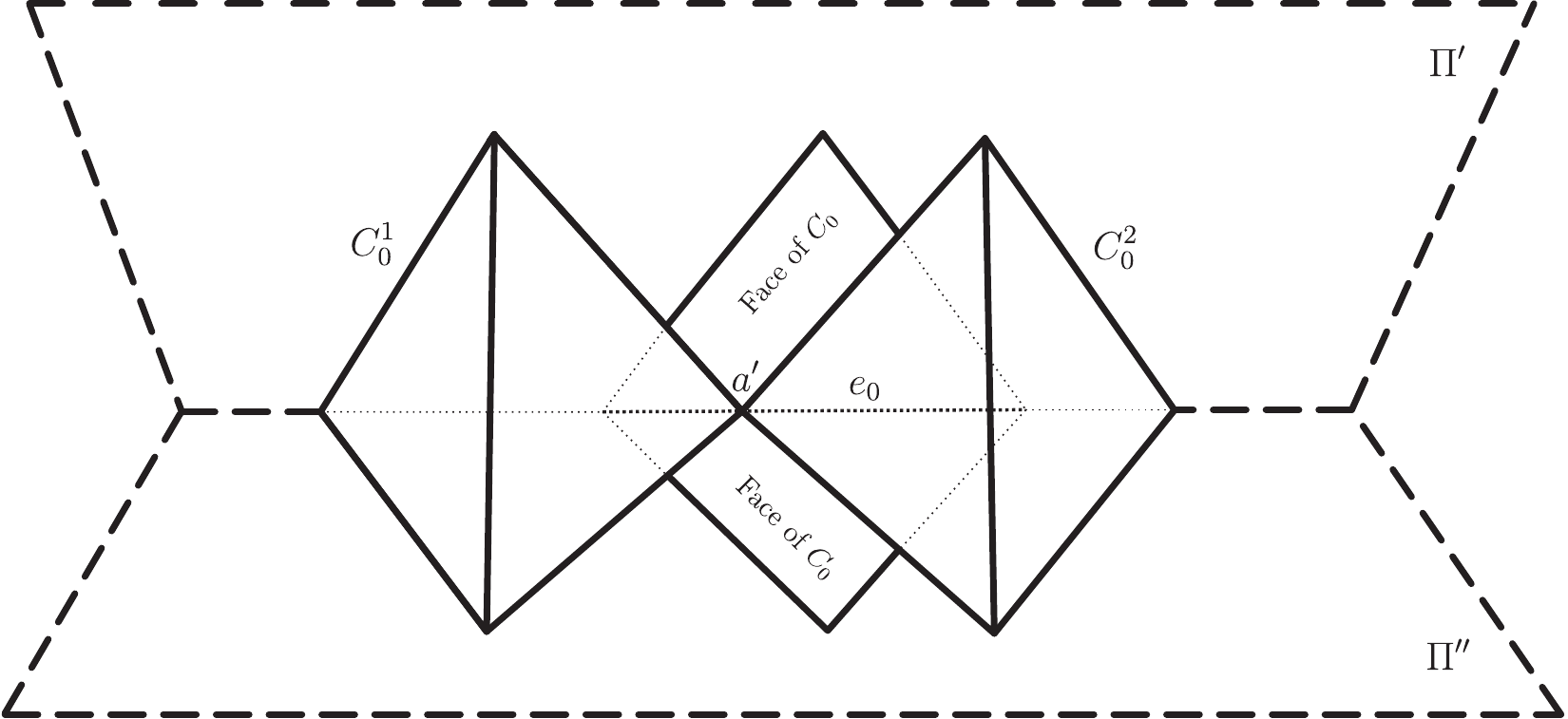}
\caption{Translates $C_0^{j}$ of $C_0$, $j= 1,2$, in Case III, \hskip2pt $d=4$.}
\label{fig9}
\end{figure}

Consider the translates $C_0^{1}$ and $C_0^{2}$ of $C_0$ along $e_0$ which have one vertex at $a'$ (see Figure~\ref{fig9}).
Both $C_0^{1}$ and $C_0^{2}$ have three edges with endpoint $a'$: one in $\Pi' \cap \Pi''$, one in $\Pi'$ and one in $\Pi''$.
These edges form angles with the plane $T$ which are $\beta_1 = 0$, \hskip3pt $\beta_2 > 0$ and $\beta_3 > 0$.
For $h \in (0, r+(m_1-a_1))$ we have dilations $C_h$, \hskip3pt $C_h^{1}$ and $C_h^{2}$ of $C_0$, $C_0^{1}$ and $C_0^{2}$,
where the latter two dilations are constructed as in Case II.
    
Following Case II, we introduce rays emanating from $a'$ along the three edges of  $C_0^{1}$ and $C_0^{2}$
and segments of lengths $p_h^i,\: i = 1,2,3$. These $p_h^i$ will satisfy \eqref{phi}.
At least one of the intersections $C_0 \cap C_0^{1}$ and $C_0 \cap C_0^{2}$ must have volume comparable to that of $C_0$.
The argument leading to \eqref{eq:p3r} can be applied to the corresponding $C_0^{j}$; cf.\ \eqref{lower}.
This gives the necessary lower estimate for $\nu(\widetilde B)$.
  
To get the corresponding upper estimate, we follow the pattern of Case~II.
We will cover $C_0$ by a finite number of (doubled) tetrahedra having one vertex at $a'$, among them
$C_0^{1}$ and $C_0^{2}$ doubled. This is done as follows.

Consider the wedge defined as that component of $\mathbb{R}^3\setminus \left(\Pi' \cup \Pi'' \right)$
which contains $C_0$. There is then a half-plane that splits this wedge in two congruent wedges denoted
$V'$ and $V''$; of these $V'$ shall be the one with boundary along $\Pi'$.

We will next rotate $C_0^{1}$, using as rotation axis the normal through $a'$ of the plane $\Pi'$.
The rotation angle will go from $0$ to $2\pi/3$; the angle $2\pi/3$ will bring $C_0^{1}$ to $C_0^{2}$.
During this rotation, the edge of $C_0^{1}$ from $a'$ in $\Pi' \cap  \Pi''$ and that in  $\Pi'$
will both stay in $\Pi'$. The edge from $a'$ which is in  $\Pi''$ before the rotation will describe a conic
surface, and its angle with $\Pi''$ will be positive and increase until it reaches a maximum at the rotation angle $\pi/3$.
Then it will decrease back to $0$. This maximum is seen to be
$2\gamma - \kappa$, and one has $0 < 2\gamma - \kappa < \gamma$, the last inequality since $\kappa > \gamma$.

This implies that the rotations of  $C_0^{1}$ considered will together cover the intersection of
$V'$ with a neighborhood of $a'$. We can then select a finite number of these rotated tetrahedra,
say $C_0^{j},\: j= 1,\dots N$, which together also cover a neighborhood of $a'$ in $V'$. Notice that $C_0^{1}$ and $C_0^{2}$
are included here. As in Case II, we consider the doubled tetrahedra $C_0^{j,2}$ with a vertex at $a'$ and conclude that 
\begin{equation*}
   \bigcup_{j=1}^N \, C_0^{j,2} \supset C_0 \cap V'.
\end{equation*}

To deal similarly with   $V''$, we repeat the rotation procedure, swapping $\Pi'$ and $\Pi''$ as well as $V'$ and $V''$.

The result will be that we cover $C_0$ by a finite number of tetrahedra, each having a vertex at $a'$.
The edges of these tetrahedra will have angles with $T$ which are larger than or equal to
$\beta_1$, \hskip3pt $\beta_2$ and $\beta_3$, respectively. This makes it possible to
argue as in Cases I and II, considering dilations $C_h$ and $C_h^{j,2}$ for $h \in (0, r+(m_1-a_1))$
and also minimal parallelepipeds.

This ends Case III and the argument in dimension four.


\end{document}